\documentclass[reqno]{amsart}
\usepackage{amsfonts,amsthm,amsmath,amssymb}
\usepackage{mathtools}
\mathtoolsset{showonlyrefs} 

\usepackage{aliascnt} 

\usepackage[dvipsnames]{xcolor}

\usepackage{tabularx} 
\usepackage{empheq}

\theoremstyle{plain}

\newtheorem{Th}{Theorem}[section]

\newtheorem*{Th*}{Theorem}

\newaliascnt{Lemma}{Th}
\newtheorem{Lemma}[Lemma]{Lemma}
\aliascntresetthe{Lemma}

\newtheorem*{Lemma*}{Lemma}

\newaliascnt{Prop}{Th}
\newtheorem{Prop}[Prop]{Proposition}
\aliascntresetthe{Prop}

\newtheorem*{Prop*}{Proposition}

\newaliascnt{Cor}{Th}
\newtheorem{Cor}[Cor]{Corollary}
\aliascntresetthe{Cor}

\newtheorem*{Cor*}{Corollary}

\theoremstyle{definition}

\newaliascnt{Def}{Th}

\aliascntresetthe{Def}

\newtheorem*{Def*}{Definition}

\newaliascnt{Ex}{Th}

\aliascntresetthe{Ex}

\newtheorem*{Ex*}{Example}

\newaliascnt{Conj}{Th}

\aliascntresetthe{Conj}

\newtheorem*{Conj*}{Conjecture}

\theoremstyle{remark}

\newaliascnt{Rem}{Th}
\newtheorem{Rem}[Rem]{Remark}
\aliascntresetthe{Rem}

\newtheorem*{Rem*}{Remark}
\let\temp\phi
\let\phi\varphi
\let\varphi\temp

\let\temp\epsilon
\let\epsilon\varepsilon
\let\varepsilon\temp

\let\subset\subseteq

\renewcommand{\hat}{\widehat}
\DeclarePairedDelimiter{\abs}{\lvert}{\rvert}
\DeclarePairedDelimiter{\ang}{\langle}{\rangle}
\DeclarePairedDelimiterX{\ip}[2]{\langle}{\rangle}{#1,#2}
\DeclarePairedDelimiter{\bra}{[}{]}
\DeclarePairedDelimiter{\norm}{\lVert}{\rVert}
\DeclarePairedDelimiter{\p}{(}{)}
\DeclarePairedDelimiter{\set}{\{}{\}}

\newcommand{\R}{\mathbb{R}}

\newcommand{\ocal}{\mathcal{O}}

\newcommand{\Hh}{\mathbf{H}}

\newcommand{\gtc}[1]{G^{#1}_\tau}

\newcommand{\pmt}{\partial M_{\tau}}

\newcommand{\pdp}{\Pi_\tau D_{\sqrt{\rho}} \Pi_\tau}
\newcommand{\pcl}{\Pi_{\chi, \lambda}}
\newcommand{\Pcl}{P_{\chi, \lambda}}
\newcommand{\Pcm}{P_{\chi, \mu}}

\newcommand{\phitilde}{\tilde{\phi}^{\C}_{\mu}}

\newcommand{\ptau}{\phi_\tau}

\newcommand{\tl}{\frac{\theta}{\lambda}}
\newcommand{\phil}{\frac{\varphi}{\lambda}}
\newcommand{\vl}{\frac{v}{\sqrt{\lambda}}}
\newcommand{\ul}{\frac{u}{\sqrt{\lambda}}}

\newcommand{\quaddifferential}{ d\sigma_1 d\sigma_2 d\mu_{\tau}(w) dt}
\newcommand{\Pil}{\Pi_{[\lambda, \lambda + 1]}}

\DeclareMathOperator{\Imp}{Im}
\DeclareMathOperator{\Rep}{Re}
\DeclareMathOperator{\supp}{supp}

\DeclareMathOperator{\vol}{vol}

\newcommand{\dtilde}[1]{\tilde{\raisebox{0pt}[0.9\height]{$\tilde{#1}$}}}
\usepackage[unicode=true]{hyperref}
\hypersetup{
    colorlinks,
    linkcolor={red!65!black},
    citecolor={purple!40!black},
    urlcolor={blue!80!black}
}

\let\C\relax
\newcommand{\C}{\mathbb{C}} 

\begin{document}

\title[]{Szeg\H{o} kernel asymptotics and concentration of Husimi Distributions of eigenfunctions}

\author{Robert Chang}
\address[Robert Chang]{Northeastern University, Boston, MA 02115}
\email{hs.chang@northeastern.edu}

\author{Abraham Rabinowitz}
\address[Abraham Rabinowitz]{Northwestern University, Evanston, IL 60208, USA}
\email{arabin@math.northwestern.edu}

\begin{abstract}
We work on the boundary $\partial M_\tau$ of a Grauert tube of a closed, real analytic Riemannian manifold \(M\). The Toeplitz operator \(\pdp\)  associated to the Reeb vector field is a positive, self-adjoint, elliptic operator on $H^2(\pmt)$. We compute $\lambda \to \infty$ asymptotics under parabolic rescaling in a neighborhood of the geodesic (Reeb) flow $\gtc{t} = \exp t\Xi_{\sqrt{\rho}}$ for the spectral projection kernel $\pcl$ associated to \(\pdp\). We also compute scaling asymptotics for tempered sums of Husimi distributions (analytic continuations) on $\pmt$ of Laplace eigenfunctions on \(M\). Both asymptotic formul\ae can be expressed in terms of the metaplectic representation of the linearization of the geodesic flow $\gtc{t}$ on Bargmann--Fock space. As a corollary, we obtain sharp \(L^p \to L^{q}\) norm estimates for $\pcl$ and sharp $L^p$ estimates for Husimi distributions.
\end{abstract}

\maketitle

\tableofcontents

\section{Statement of main results}

The main purpose of this article is to study the $L^p \to L^q$ mapping norms of the spectral projections $\pcl$ \eqref{eqn:PCL} and $\Pi_{[\lambda,\lambda+1]}$ \eqref{eqn:SHORTWINDOW} associated to the Szeg\H{o} projector $\Pi_\tau$ \eqref{eqn:SZEGOPROJ} on the boundary $\pmt$ a of Grauert tube $M_\tau$. These norm estimates, which are sharp, are stated in \autoref{theo:MAINTHEO} and \autoref{Short window}. A key ingredient of the proof is the on-shell off-diagonal scaling asymptotics of $\pcl$ on $\pmt$ (\autoref{Graph Scaling Theorem}). As applications, we deduce sharp $L^p$ estimates for analytic continuations \eqref{eqn:HUSIMI} of Laplace eigenfunctions (\autoref{theo:MAINTHEO}), as well as for eigenfunctions of the Toeplitz operator \eqref{eqn:ELJ}, whose principal symbol coincides with that of $\sqrt{-\Delta}$ transported to the Grauert tube boundary (\autoref{prop:Approx Reproduce}). Unlike the Sogge estimates in the real domain, there is no `critical exponent' $p$ separating low and high $L^p$ norms.

We now give a terse introduction to and the precise statements of our results, postponing  to \autoref{sec:COMP} a more detailed discussion of related works. The setup involves a closed, real analytic manifold  \((M,g)\) of dimension \(m \geq 2\). Its complexification \(M_\C\) admits a strictly plurisubharmonic exhaustion function \(\rho\) in a neighborhood of the totally real submanifold \(M \subset M_\C\). For each $0 < \tau < \tau_{\mathrm{max}}$, the sublevel set \(\{\sqrt{\rho} < \tau\} =: M_\tau\) is a K\"{a}hler manifold, called the \emph{Grauert tube of radius $\tau$}.

Throughout, we work on the boundary of a Grauert tube with a fixed radius $\tau$. The \emph{Szeg\H{o} projector}  associated the Grauert tube boundary
\begin{equation}\label{eqn:SZEGOPROJ}
\Pi_\tau \colon L^2(\partial M_\tau) \to H^2(\partial M_\tau)
\end{equation}
is the orthogonal projection onto the Hardy space of boundary values of holomorphic functions in the tube. Consider the Toeplitz operator
\begin{equation}\label{eqn:TOEPLITZ}
\Pi_\tau D_{\sqrt{\rho}} \Pi_\tau \colon H^2(\partial M_\tau) \to H^2(\partial M_\tau),
\end{equation}
where $D_{\sqrt{\rho}} = \frac{1}{i} \Xi_{\sqrt{\rho}}$ is a constant multiple of the Hamilton vector field of the Grauert tube function $\sqrt{\rho}$ acting as a differential operator. As in \cite{ChangRabinowitz21}, we fix a positive, even Schwartz function $\chi$ whose Fourier transform is compactly supported with \( \widehat{\chi}(0) = 1\) and form the spectral localization
\begin{equation}\label{eqn:PCL}
\Pi_{\chi,\lambda} =  \int_\R \widehat{\chi}(t) e^{-it\lambda} \Pi_\tau e^{it \Pi_\tau D_{\sqrt{\rho}}\Pi_\tau}\,dt.
\end{equation}
The classical dynamics associated to \(e^{i t \Pi_\tau D_{\sqrt{\rho}} \Pi_\tau}\) on \(\pmt\) is the Hamilton flow
\begin{align}\label{eqn:GTC}
\gtc{t} \colon \pmt \to \pmt, \quad \gtc{t} = \exp t\Xi_{\sqrt{\rho}}.
\end{align}
This flow coincides with the pullback of the Riemannian geodesic flow on $S^*_\tau M$ under the diffeomorphism \eqref{eqn:DIFFEO}.

Our scaling asymptotics for \eqref{eqn:PCL} is stated in Heisenberg coordinates in the sense of \cite{FollandStein} centered at \(p \in \pmt\) and \(\gtc{s}(p)\in \pmt\). In these coordinates, the derivative of the flow \eqref{eqn:GTC} takes the form 
\begin{align}\label{eqn:SYMPMAT}
  D \gtc{s}: T_p\pmt \to T_{\gtc{s}(p)}\pmt, \qquad   D\gtc{s} = \begin{pmatrix} 1 & 0 \\ 0 & M_s \end{pmatrix},
\end{align}
where \(M_s \) is a symplectic matrix on \(\R^{2(m-1)}\). Let \(\hat{\Pi}_{\mathcal{H}, M_s}\) denote the lift to the reduced Heisenberg group $\Hh_\mathrm{red}^{m-1} = S^1 \times \C^{m-1}$ of the metaplectic representation of \(M_s\) acting on the model Bargmann--Fock space \(\mathcal{H}(\C^{m-1})\). See \autoref{sec:METAPLECTIC} for details. The following theorem states that under a parabolic \(\lambda\)-rescaling near \(p\) and \(\gtc{s}(p)\), the kernel of \eqref{eqn:PCL} behaves like  \( \lambda^{m-1}\hat{\Pi}_{\mathcal{H},M_s}\) to leading order as \(\lambda \to \infty\).

%
%
%
%
%

%
%
%

\begin{Th}[On-shell scaling asymptotics for $\pcl$]\label{Graph Scaling Theorem}
Let $M$ be a closed, real analytic Riemannian manifold of dimension $m \geq 2$. Let $\pcl$ be as in \eqref{eqn:PCL} and $\gtc{t}$ be as in \eqref{eqn:GTC}. Fix \(p \in \pmt\) and \(s \in \supp \hat{\chi}\). Let $(\theta, u)$ and $(\varphi, v)$ be Heisenberg coordinates centered at $p$ and \(\gtc{s}(p)\), respectively. Then, we have
\begin{align}
\Pi_{\chi, \lambda}&\left( p +  \left(\frac{\theta}{\lambda}, \frac{u}{\sqrt{\lambda}}\right);\gtc{s}(p) + \left( \frac{\varphi}{\lambda}, \frac{v}{\sqrt{\lambda} }\right) \right) \\
&=  \frac{C_{m,M}}{\tau^{m}}e^{i s \lambda} \lambda^{m-1}\hat{\Pi}_{\mathcal{H},M_s}\left(\frac{\theta}{2 \tau}, \frac{u}{\sqrt{\tau} }, \frac{\varphi}{2 \tau}, \frac{v}{\sqrt{\tau} }\right)  \bra*{1 + \sum_{j = 1}^{N}\lambda^{ - \frac{j}{2}}P_j(p,s,u,v,\theta,\varphi)}\\
& \quad +  \lambda^{^{ - \frac{N + 1}{2}}} R_N\left(p,s, \theta, u , \varphi, v ,\lambda \right),
\end{align}
where \(P_j\) is a polynomial in \(\theta,u, \varphi, v\), the remainder $R_N$ satisfies
\begin{align}
\norm{R_{N}( p,s, \theta, u, \varphi, v, \lambda )}_{C^{j}\left( \{|(\theta,u)| + |\left( \varphi,v \right)| \leq \rho\}  \right)} \leq C_{N,j,\rho} \quad \text{for } \rho > 0 , \ j = 1,2,3, \dotsc ,
\end{align}
and all quantities vary smoothly with \(p\) and \(s\).
\end{Th}

\begin{Rem}
  When $s = 0 $ so that $M_s = I$ is the identity matrix,
\begin{align}\label{eqn:PILIFT}
\hat{\Pi}_{\mathcal{H}, I}(\theta, u; \varphi, v) = \frac{1}{\pi^{m-1}} e^{i(\theta-\varphi) + u \cdot \bar{v} - \tfrac{1}{2}\abs{u}^2 - \tfrac{1}{2}\abs{v}^2}
\end{align}
coincides with the Szeg\H{o} kernel of level one on $\Hh_\mathrm{red}^{m-1}$, and we recover near diagonal asymptotics computed in \cite{ChangRabinowitz21} .
\end{Rem}

An argument similar to that found in \cite{ShiffmanZelditch03} allows us to deduce the following sharp $L^p \to L^q$ mapping norm estimate for \eqref{eqn:PCL}.

\begin{Th}[$L^p \to L^q$ mapping estimate for $\pcl$]\label{theo:MAINTHEO}
Let $M$ be a closed, real analytic Riemannian manifold of dimension $m \geq 2$. Let $\pcl$ be as in \eqref{eqn:PCL}.  Then we have the sharp estimate
\begin{align}
  \norm{\pcl f}_{L^{q}(\pmt)} \leq C_{\pmt} \lambda^{(m - 1) (\frac{1}{p} - \frac{1}{q} )} \norm{ f}_{L^{p}( \pmt )} \qquad \left( 2 \leq p,q \leq \infty \right).
\end{align}
\end{Th}

We now turn to estimates for eigenfunctions. The Toeplitz operator \eqref{eqn:TOEPLITZ} is a positive, self-adjoint, elliptic operator on \(H^{2}\left( \pmt \right)\) in the sense of \cite{BoutetGuillemin81}, so has a discrete spectrum $ 0 = \lambda_0 <\lambda_1\leq  \lambda_2 \leq \dotsb$ with associated \(L^{2}\)-normalized eigenfunctions
\begin{align}\label{eqn:ELJ}
  \pdp e_{\lambda_j} = \lambda_j e_{\lambda_j}, \qquad \norm{e_{\lambda_j}}_{L^{2}\left( \pmt \right)} = 1.
\end{align}
An immediate consequence of the eigenfunction expansion
\begin{equation}
\pcl = \sum_{j : \lambda_j \le \lambda} \chi( \lambda - \lambda_j)e_{\lambda_j} \otimes \overline{e_{\lambda_j}},\label{expansion}
\end{equation}
of the spectral projection \eqref{eqn:PCL} together with \autoref{theo:MAINTHEO} in the case \(p = 2\) is the following.

\begin{Cor}[$L^p$ estimates for eigenfunctions of $\pdp$]\label{cor:LP}
Let $M$ be a closed, real analytic Riemannian manifold of dimension $m \geq 2$. Let $e_{\lambda_j}$ be $L^2$-normalized eigenfunctions of $\pdp$ as in \eqref{eqn:ELJ}. Then we have
\begin{align}
  \norm{e_{\lambda_j}}_{L^{q}\left( \pmt\right)} \leq C_{\pmt} \lambda_j^{\left( m - 1 \right)( \frac{1}{2} - \frac{1}{q} )} \qquad (2 \leq q \leq \infty)
\end{align}
\end{Cor}

The conclusion of \autoref{theo:MAINTHEO} also holds for spectral projections onto short spectral intervals
\begin{equation}\label{eqn:SHORTWINDOW}
\Pi_{[\lambda , \lambda + 1]} = \sum_{j : \lambda \le \lambda_j \le \lambda + 1} e_{\lambda_j} \otimes \overline{e_{\lambda_j}},
\end{equation}
which we state below in \autoref{Short window}. This result may be viewed as a Grauert tube analogue, with $\pdp$ replacing $\sqrt{-\Delta}$, of Sogge's $L^p$ estimate \cite{SoggeLP} for spectral projections of the Lapalcian.

\begin{Th}\label{Short window}
  Let $M$ be a closed, real analytic Riemannian manifold of dimension $m \geq 2$. Then we have the sharp estimate
  \begin{align}
    \norm{\Pi_{[\lambda, \lambda + 1]} f}_{L^{q}(\pmt)} \leq C_{\pmt} \lambda^{(m - 1) (\frac{1}{2} - \frac{1}{q} )} \norm{ f}_{L^{2}( \pmt )} \qquad ( 2 \leq q \leq \infty )
  \end{align}
\end{Th}

Our next set of results concern analytic continuations to the Grauert tube boundary of Laplace eigenfunctions on $M$. Let $0 = \mu_0 < \mu_1 \le \mu_2 \le \dotsb$ be eigenvalues of $\sqrt{-\Delta}$ with associated $L^2$-normalized eigenfunctions
\begin{align}\label{eqn:EIGENEQUATION}
    -\Delta \phi_{\mu_j} = \mu_j^{2}\phi_{\mu_j}, \qquad  \norm{\phi_{\mu_j}}_{L^{2}(M)} = 1.
\end{align}
The analytic extensions $\phi_{\mu_j}^\C$, which are CR holomorphic functions on $\pmt$, are defined by
\begin{equation}\label{eqn:UNNORMALIZED}
\phi_{\mu_j}^\C = e^{\tau \mu_j} U(i\tau) \phi_{\mu_j},
\end{equation}
where $U(i\tau) = \exp(-\tau \sqrt{-\Delta})$ is the Poisson operator. See \autoref{sec:OPERATORS} for details. The probability amplitudes
\begin{align}\label{eqn:HUSIMI}
  \tilde{\phi}^{\C}_{\mu_j} = \frac{\phi^{\C}_{\mu_j}}{\norm{\phi^{\C}_{\mu_j}}_{L^{2}( \pmt )}}
\end{align}
are \emph{Husimi distributions}, that is, \emph{microlocal lifts} of $\phi_{\mu_j}$ to phase space $\pmt \cong S^*_\tau M$. They are ``approximate eigenfunctions'' of $\pdp$ in the following sense.

\begin{Prop}\label{prop:Approx Reproduce}
  Let \(\{\lambda_j\}\) and \(\{\mu_j\}  \) be the eigenvalues of \(\pdp\) and \(\sqrt{-\Delta} \), respectively. Let \(\pcl\) be as in \eqref{eqn:PCL} and \(\tilde{\phi}^{\C}_{\mu_j}\) be as in \eqref{eqn:HUSIMI}. Then, \(\mu_j = \lambda_j + O(1)\) as \(j \to \infty\) and
    \begin{align}
    \norm{\pcl{\tilde{\phi}^{\C}_{\mu_j}} - \tilde{\phi}^{\C}_{\mu_j}}_{L^2(\pmt)} &= O(1),\\
    \norm{\pdp \tilde{\phi}^{\C}_{\mu_j} -   \lambda_j\tilde{\phi}^{\C}_{\mu_j}}_{L^2(\pmt)} &=  O(1).
  \end{align}
 Furthermore, we have the sharp estimate
 \begin{align}\label{eqn:HUSIMILP}
     \norm{\tilde{\phi}^{\C}_{\mu_j}}_{L^{p}( \pmt )} \leq  C_{\pmt}\lambda_j^{( m - 1 )( \frac{1}{2} - \frac{1}{q} )} \qquad (2 \leq p \leq \infty).
 \end{align}
\end{Prop}
The $L^p$ bound \eqref{eqn:HUSIMILP} may also be deduced from the sup norm bound of Zelditch \cite{Zelditch20Husimi} and log-convexity of \(L^{p}\) norms. In \autoref{sec:SHARP}, we show that the bound is saturated by complexified Gaussian beams. This result is yet another Grauert tube analogue, with the analytically continued $\tilde{\phi}_{\lambda_j}^\C$ replacing $\phi_{\lambda_j}$, of Sogge's $L^p$ estimate for eigenfunctions. Note that, unlike in the real domain, there are no separate estimates for ``high'' versus ``low'' $L^p$; see \autoref{sec:COMPREAL} for further discussion.

\begin{Rem}
In \cite{ChangRabinowitz21}, we showed that analytically continued eigenfunctions on a torus are approximate eigenfunctions of $\pdp$ by an explicit computation, and that analytically continued spherical harmonics are in fact exact eigenfunctions of $\pdp$.
\end{Rem}

To state the last result, we fix, as before, a positive, even Schwartz function $\chi$ whose Fourier transform is compactly supported with $\hat{\chi}(0) = 1$ to construct the tempered partial sums
\begin{align} \label{Tempered Sum}
   \Pcm = \sum_{j : \mu_j \le \mu} \chi( \mu - \mu_j) e^{ - 2 \tau \mu_j}\phi^{\C}_{\mu_j} \otimes \overline{\phi^{\C}_{\mu_j}}.
\end{align}
using \eqref{eqn:UNNORMALIZED}. The prefactor $e^{-2\tau \mu_j}$ is introduced to ``temper'' the exponential growth estimate (\cite[Corollary~3]{Zelditch12potential}) for complexified eigenfunctions:
\begin{equation}
C_1 \mu_j^{-\frac{m-1}{2}} e^{\tau \mu_j} \le \norm{\phi_{\mu_j}^\C(\zeta)}_{L^\infty(\pmt)} \le C_2 \mu_j^{\frac{m-1}{2}} e^{\tau \mu_j}.
\end{equation}

\begin{Rem}
It is shown in \cite{Lebeau,Zelditch12potential} that  \(e^{ - \tau \mu_j}\phi_{\mu_j}^{\C}\) is a Riesz basis (but not an orthonormal basis) in general, so they fail to be reproduced by $\Pcm$.
\end{Rem}

The proof of \autoref{Graph Scaling Theorem} is easily adapted to prove scaling asymptotics for \eqref{Tempered Sum}. Comparing the statements of \autoref{Graph Scaling Theorem} and \autoref{Graph Scaling Theorem2} below, we see the leading order asymptotics of $\pcl$ and $\Pcm$ differ only in the powers of the frequency parameters $\lambda$ or $\mu$.

\begin{Th}[On-shell asymptotics for $\Pcm$]\label{Graph Scaling Theorem2}
		Let $\Pcm$ be as in \eqref{eqn:PCL}. Under the same hypotheses as \autoref{Graph Scaling Theorem}, we have
		\begin{align}
		\Pcm&\left( p +  \left(\frac{\theta}{\mu}, \frac{u}{\sqrt{\mu}}\right);\gtc{s}(p) + \left( \frac{\varphi}{\mu}, \frac{v}{\sqrt{\mu} }\right) \right) \\
&=  \frac{C_{m,M}}{\tau^{m}}e^{i s \mu} \mu^{\frac{m-1}{2}}\hat{\Pi}_{\mathcal{H},M_s}\left(\frac{\theta}{2 \tau}, \frac{u}{\sqrt{\tau} }, \frac{\varphi}{2 \tau}, \frac{v}{\sqrt{\tau} }\right)  \bra*{1 + \sum_{j = 1}^{N}\mu^{ - \frac{j}{2}}P_j(p,s,u,v,\theta,\varphi)}\\
& \quad +  \mu^{^{ - \frac{N + 1}{2}}} R_N\left(p,s, \theta, u , \varphi, v ,\mu \right),
\end{align}
where \(P_j\) is a polynomial in \(\theta,u, \varphi, v\), the remainder $R_N$ satisfies
\begin{align}
\norm{R_{N}( p,s, \theta, u, \varphi, v, \mu )}_{C^{j}\left( \{|(\theta,u)| + |\left( \varphi,v \right)| \leq \rho\}  \right)} \leq C_{N,j,\rho} \quad \text{for } \rho > 0 , \ j = 1,2,3, \dotsc ,
\end{align}
and all quantities vary smoothly with \(p\) and \(s\).
\end{Th}

\begin{Rem}
Our techniques for proving \autoref{Graph Scaling Theorem} and \autoref{theo:MAINTHEO} hold in the more general setting of a compact, strictly pseudoconvex CR manifold $X$ for which $\Box_b$ has closed range. In particular, the Boutet de Monvel--Sj\"{o}strand description of the Szeg\H{o} projector remains valid and quantization of the geodesic flow \eqref{eqn:GTC} can be replaced by that of the Reeb flow. The main interest in the Grauert tube setting is the manifestation of the underlying Riemannian geometry as well as the connection between analytic extensions and microlocal lifts of eigenfunctions.
\end{Rem}

\subsection{Comparison to prior results}\label{sec:COMP}

In \autoref{sec:FOURIERANALOGUE}, we briefly justify the operator \eqref{eqn:PCL} as a Grauert tube analogue of the Bergman projections in the line bundle setting; a more detailed discussion is contained in \cite[Section~1.1]{ChangRabinowitz21}. \autoref{sec:LPLINEBUNDLE} recalls some results on $L^p$ estimates on Bergman kernels associated to line bundles, as well as asymptotic expansions of quantized Hamiltonian symplectomorphisms. We return to the real domain in \autoref{sec:COMPREAL} with a comparison of $L^p$ norms of eigenfunctions on the manifold $M$ versus those of analytically continued eigenfunctions on the tube $\pmt$.

\subsubsection{The operators \texorpdfstring{$\pcl$}{Pi} as Fourier components of the Grauert tube Szeg\H{o} kernel}\label{sec:FOURIERANALOGUE}

Let $(L, h) \to (X, \omega)$ be a positive Hermitian line bundle over a closed K\"{a}hler manifold. Let $\partial D = \set{\ell \in L^* : \norm{\ell}_{h^*} = 1}$ be the unit co-circle bundle. The orthogonal projections
\begin{equation}\label{eqn:BERGMANSZEGOLINE}
\Pi_{h^k} \colon L^2(X, L^k) \to H^0(X, L^k) \quad \text{and} \quad \Pi_h \colon L^2(\partial D) \to H^2(\partial D)
\end{equation}
are the \emph{Bergman} and \emph{Szeg\H{o}} projections, respectively. The space $H^0(X, L^k)$ of holomorphic sections is unitarily equivalent to the set of equivariant CR functions $f \in H^2(\partial D)$ satisfying $f(r_\theta x) = e^{ik\theta}f(x)$, where  $r_\theta$ denotes the circle action on $\partial D$. Under this identification, operators \eqref{eqn:BERGMANSZEGOLINE} are related by
\begin{equation}\label{eqn:SPECTRALDECOMP}
\Pi_{h^k} = \frac{1}{2\pi}\int_0^{2\pi} e^{-i k \theta} (r_\theta)^*\Pi_h \,d\theta.
\end{equation}
Note that the Fourier decomposition of $\Pi_h$ coincides with the spectral decomposition of $D_\theta = \frac{1}{i} \frac{\partial}{\partial \theta}$ on $\partial D$.

Turning to the Grauert tube setting of \eqref{eqn:SZEGOPROJ} and \eqref{eqn:GTC}, the naive approach of replacing $(r_\theta)^*\Pi_h$ by $(\gtc{t})^* \Pi_\tau$ is insufficient. Instead, the former is replaced by $\Pi_\tau \widehat{\sigma} (\gtc{t})^* \Pi_\tau$, where $\hat{\sigma}$ is a polyhomogeneous pseudodifferential operator on $\partial M_\tau$ that makes the resulting composition unitary. Thus, we are led to the CR holomorphic analogue of \eqref{eqn:SPECTRALDECOMP} that is
\begin{equation}\label{eqn:SPECTRALDECOMPANALOGUE}
\pcl := \int_\R \hat{\chi}(t) e^{-i\lambda t} \Pi_\tau \hat{\sigma}(\gtc{t})^*\Pi_\tau\,dt \sim  \int_\R \hat{\chi}(t) e^{-i\lambda t} \Pi_\tau e^{it\Pi_\tau D_{\sqrt{\rho}}\Pi_\tau}\,dt.
\end{equation}
Note that \eqref{eqn:SPECTRALDECOMPANALOGUE} is essentially the spectral decomposition of the elliptic Toeplitz operator $\Pi_\tau D_{\sqrt{\rho}}\Pi_\tau$ introduced in \eqref{eqn:TOEPLITZ}. The same operator is also studied in \cite[Theorem~0.12]{Zelditch20Husimi}.

\subsubsection{\texorpdfstring{\(L^p\)}{Lp} estimates and quantized Hamiltonians on line bundles} \label{sec:LPLINEBUNDLE}

In the line bundle setting $(L,h) \to (X,\omega)$, Shiffman--Zelditch \cite[Lemma~4.1]{ShiffmanZelditch03} proved $\norm{\Pi_{h^k}}_{L^p \to L^q} \le C k^{m(1/p - 1/q)}$ by combining the Shur--Young inequality with a near-diagonal Gaussian estimate \cite[Lemma~5.2]{ShiffmanZelditch02}. Consequently, $\norm{s}_{L^p} = O(k^{m(1/2 - 1/p)})$ for all $L^2$-normalized holomorphic sections $s \in H^0(X, L^k)$. 

The proof techniques of our Grauert tube analogue, \autoref{theo:MAINTHEO}, are similar. But, in place of a near-diagonal scaling asymptotics, we need the full strength of \autoref{Graph Scaling Theorem}, which is an asymptotic expansion in a $\lambda^{-\frac{1}{3}}$-neighborhood of the orbit $p \mapsto \gtc{s}(p)$. This finer control of the time evolution under the Reeb flow (i.e., $\gtc{t}$ on $\pmt$ or $r_\theta$ on the circle bundle) is unnecessary in the line bundle setting because rotations of the fiber introduce only an overall phase factor to the near-diagonal scaling asymptotics.

There has also been prior work on quantized Hamiltonian flows on line bundles. More precisely, let \(f \in C^{\infty}(X)\) be a Hamiltonian on the classical phase space (K\"{a}hler manifold) $X$ that induces a \(1\)-parameter group of symplectomorphisms \(\phi_t \colon X \to X\) which lifts to a family of contactomorphisms \(\tilde{\phi}_t \colon \partial D \to \partial D\). As shown by Zelditch \cite{Zelditch97index}, these contactomorphisms may be quantized as unitary maps
\begin{align}
\Phi_t \colon L^2 ( \partial D ) \to L^2( \partial D ), \quad  \Phi_t =  R_t  \Pi_h ( \tilde{\phi}_{ - t} )^{*}  \Pi_h,
\end{align}
in which \(R_{ - t}\) is a zeroth order Toeplitz operator chosen to ensure the unitarity of \(\Phi_t\). In a series of papers, Paoletti \cite{Paoletti20122,Paoletti20121,Paoletti2014} computed scaling asymptotics for the Fourier coefficients (with respect to the $S^1$ action) of \(\Phi_t\) near points on the graph of \(\tilde{\phi}_t\). When \(t = 0 \), this specializes to the scaling asymptotics of \cite{BleherShiffmanZelditch00,ShiffmanZelditch02}. We emphasize that in contrast, \(\gtc{t}\) is simultaneously playing the role of the \(S^1\) action and the Hamiltonian flow in our set up. Nevertheless, our main argument borrows heavily from that of \cite{Paoletti20121}. We also mention the works of Zelditch--Zhou \cite{ZelditchZhou2018,ZelditchZhou20191,ZelditchZhou20192},  which treat other types of asymptotics for partial Bergman kernels of quantized Hamiltonian flows on line bundles. 

\subsubsection{\texorpdfstring{$L^p$}{Lp} norms of eigenfunctions in the real domain}\label{sec:COMPREAL}
In this section we discuss how the \(L^{p}\) estimates of \autoref{theo:MAINTHEO} and \autoref{cor:LP} compare with those of Sogge eigenfunctions in the real domain. Let \(P_{[\lambda, \lambda + 1]}\) denote the orthogonal projection onto the span of Laplace eigenfunctions \(\phi_{\lambda_j}\) with frequencies \(\lambda \leq \lambda_j < \lambda + 1\).

\begin{Th}[{Sogge \cite{SoggeLP}, see also \cite{SoggeFICA,KochTataruZworski}}]
  Let \((M, g)\) be a closed Riemannian manifold of dimension \(n\), then the following estimates are sharp.
  \begin{align}
    \norm{P_{[\lambda, \lambda + 1]}f}_{L^{q}( M )} \leq \begin{cases} C \lambda^{(\frac{n - 1}{2}) ( \frac{1}{2} - \frac{1}{q} )} \norm{f}_{L^{2}( M )} & \text{for $2 \leq q \leq \frac{2(n + 1)}{n - 1}$},\\
     C \lambda^{n( \frac{1}{2} - \frac{1}{q} ) - \frac{1}{2}} \norm{f}_{L^{2}( M )} & \text{for $\frac{2(n + 1)}{n - 1} \leq q < \infty$}.
   \end{cases} 
  \end{align}
  Consequently, for $L^2$-normalized Laplace eigenfunctions \(\phi_{\lambda_j}\) with frequencies $\lambda_j$, we have
  \begin{align}
     \norm{\phi_{\lambda_j}}_{L^{q}( M )} \leq \begin{cases} C \lambda_j^{(\frac{n - 1}{2}) ( \frac{1}{2} - \frac{1}{q} )}  & \text{for $2 \leq q \leq \frac{2(n + 1)}{n - 1}$},\\
     C \lambda_j^{n( \frac{1}{2} - \frac{1}{q} ) - \frac{1}{2}}  & \text{for $\frac{2(n + 1)}{n - 1} \leq q < \infty$}. 
   \end{cases}
  \end{align}
 \end{Th}
Note the presence of a critical exponent \(q_n = \frac{2(n + 1)}{n - 1}\) at which the sharp estimates change. Roughly speaking, high \(L^{q}\) norms measure concentration around single points, whereas low \(L^{q}\) norms measure concentration around larger sets such as geodesics and hypersurfaces. It is well known on the round sphere \(S^{n}\) that the sequence of zonal spherical harmonics at a pole saturate the estimate for \(q > q_n\). On the other hand, the sequence of highest weight spherical harmonics, that is Gaussian beams along a stable elliptic geodesic, saturate the estimate for \(q < q_n\). However, these bounds are rarely sharp on other manifolds. For example, on the flat torus all eigenfunctions have \(L^{q}\) norms bounded by \(O(1)\). An interesting question in this direction is which manifolds admit sequences of eigenfunctions that saturate these bounds. We point the readers to \cite{SoggeTothZelditch,SoggeZelditchFocal1,SoggeZelditchFocal2,SoggeZelditchMax} for research in this topic.

Although we project onto the orthonormal basis consisting of eigenfunctions of \(\pdp\) rather than onto the span \(\tilde{\phi}_{\lambda_j}^{\C}\), thanks to \autoref{prop:Approx Reproduce} we can interpret \autoref{Short window} as a complexified version of the theorem above. In the complex setting there is no critical exponent \(q_n\) differentiating the behavior between the low and high \(L^{q}\) norms. Indeed, the exponent in our sharp estimate is analogous to that of Sogge's in the low \(L^{q}\) regime, and we show in \autoref{sec:SHARP} that complexified Gaussian beams are extremals for all \(p\).

Our main result has an interpretation as measuring concentration in phase space. As discussed in \cite{Zelditch20Husimi}, the squares of \(\tilde{\phi}^{\C}_{\lambda_j}\) are microlocal lifts of \(\phi_{\lambda_j}\) to  \(\pmt \cong S^{*}_{\tau}M\), so they may be viewed as  probability densities of finding a quantum particle at a phase space point in \(\pmt\). Their marginals are given by the pushforward \(\pi_{*}(\tilde{\phi}^{\C}_{\lambda_j})^2\) under the natural projection \(\pi \colon S^{*}_{\tau}M \to M\). It is natural to ask how the marginal densities of these Husimi distributions relate to eigenfunction concentration on $M$. Other types of phase space norms of eigenfunctions have been studied by Blair--Sogge \cite{BlairSogge1,BlairSogge2}. We also mention the work of Galkowski \cite{Galkowski}, which uses defect measures to study eigenfunction concentration. It would be interesting to compare the results and techniques with those of complexification.

\subsection{Organization of the paper}\label{sec:ORG}
\autoref{sec:BACKGROUND} collects the relevant facts pertaining to Grauert tubes. We recall the Szeg\H{o} projector as a complex Fourier integral operator (FIO) with a positive complex canonical relation as well as the Boutet de Monvel--Sj\"{o}strand parametrix for the kernel. We also recall the relevant microlocal properties of the Poisson wave operator \eqref{eqn:POISSON} and the tempered spectral projection \eqref{eqn:TEMPEREDPROJ} studied in  \cite{Zelditch07complex,Zelditch12potential,Zelditch20Husimi}. Also in \autoref{sec:BACKGROUND} is a brief review the metaplectic representation on the Bargmann--Fock space of \(\C^{n}\).

The proofs of the two scaling asymptotics are contained in \autoref{sec:PROOFSCALING}, while the $L^p$ estimates for the projector and for eigenfunctions are found in \autoref{sec:LP}. In section 6 we prove the \autoref{theo:MAINTHEO}, \autoref{Short window}, and \autoref{prop:Approx Reproduce}. Finally in section 7 we demonstrate that Gaussian beams on the sphere saturate \(L^{p}\) bounds and give a geometric explanation.

\subsection{Acknowledgment}
The authors would like to thank Steve Zelditch for bringing our attention to Grauert Tubes and many helpful discussions in the writing of this article.

\section{Background}\label{sec:BACKGROUND}

We assume throughout that \((M,g)\) is a closed, real analytic Riemannian manifold of dimension $m \geq 2$.  Readers may consult \cite{GuilleminStenzel91,GuilleminStenzel92,LempertSzoke91,LempertSzoke01,GolseLeichtnamStenzel96}  for geometry of and analysis on Grauert tubes (particularly in relation to the complex Monge--Amp\`{e}re equation and complexified geodesics), as well as a paper \cite{ChangRabinowitz21} of the authors with a more detailed discussion.

\subsection{K\"{a}hler geometry on Grauert tubes}\label{sec:GTUBE}

A real analytic Riemannian manifold $M$ admits a complexification \(M_{\C}\) into which $M$ embeds as a totally real submanifold. The \emph{Grauert tube function} is defined by
\begin{equation}\label{eqn:GRAUERTTUBEFUNCTION}
    \sqrt{\rho} \colon U \subset M_\C \to \R, \quad \sqrt{\rho}(z) = \frac{1}{2i} \sqrt{r^2_\C(z, \overline{z})},
\end{equation}
where  \(r^2_\C(z,\overline{w})\) is the analytic extension of the square of the Riemannian distance function \(r \colon M \times M \to \R\) to a neighborhood of the diagonal in $M_\C \times \overline{M}_\C$. In a neighborhood of $M$ in $M_\C$, the square \(\rho\)  of \eqref{eqn:GRAUERTTUBEFUNCTION} is the unique strictly plurisubharmonic function such that the metric induced by the K\"{a}hler form \(i \partial \overline{\partial} \rho \) restricts to the Riemannian metric \(g\) on \(M\). 

For each \(0 < \tau \le \tau_{\mathrm{max}}\), the sublevel set
\begin{equation}\label{eqn:MTAU}
    M_{\tau} = \{z \in M_\C : \sqrt{\rho}(z) < \tau\}
\end{equation}
is called the \emph{Grauert tube of radius \(\tau\)}. It is diffeomorphic \cite[Theorem~1.5]{GolseLeichtnamStenzel96} to the co-ball bundle $B^*_\tau M = \{(x,\xi) \in T^*M : \abs{\xi}_x < \tau\}$ of radius $\tau$ under the imaginary-time exponential map
\begin{equation}\label{eqn:DIFFEO}
    E \colon B^{*}_{\tau}M \to M_\tau, \quad E(x,\xi) = \exp_x^\C i\xi.
\end{equation}

Let \(G^t\) denote the homogeneous geodesic flow, that is, the Hamilton flow of \(|\xi|_x\), on the cotangent bundle. Then, for each $0 < \tau \le \tau_{\mathrm{max}}$, the $C^\omega$ diffeomorphism \eqref{eqn:DIFFEO} conjugates the geodesic flow on the sphere bundle $S^*_\tau M$ to the Hamilton flow $\gtc{t} = \exp t\Xi_{\sqrt{\rho}}$ of the Grauert tube function on $\partial M_\tau$:
\begin{equation}\label{eqn:GTTAU}
G^t_\tau \colon \partial M_\tau \to \partial M_\tau, \quad G^t_\tau = E \circ G^t \circ E^{-1}|_{\partial M_\tau}.
\end{equation}
With this identification in mind, we will henceforth refer to $\gtc{t} \colon \pmt \to \pmt$ as the ``geodesic flow.''

\subsection{Contact and CR structure on the Grauert tube boundary}\label{sec:CONTACT} 

The Grauert tube $M_\tau$ is a strongly pseudoconvex domain thanks to the existence of the strictly plurisubharmonic exhaustion function $\sqrt{\rho}$. The pullbacks of the canonical 1-form $\alpha_{T^*M} = \xi \,dx$ and the symplectic form $\omega_{T^*M} = d\xi \wedge dx$ on the cotangent bundle under the diffeomorphism \eqref{eqn:DIFFEO} are
\begin{equation}\label{eqn:EQUAL}
\alpha:= (E^{-1})^* \alpha_{T^*M} = d^c\sqrt{\rho} \quad \text{and} \quad (E^{-1})^*\omega_{T^*M} = dd^c \rho.
\end{equation}

We endow the Grauert tube boundary $\pmt$ with the volume form
\begin{equation}\label{eqn:CONTACTVOLUME}
d \mu_{\tau} = (E^{-1})^*(\alpha_{T^*M} \wedge \omega_{T^*M}^{m-1})\Big|_{\partial M_\tau},
\end{equation}
which is the pullback of the standard Liouville volume form on $S^*_\tau M$. Since $\pmt$ is a real hypersurface in the K\"{a}hler manifold $M_{\tau_{\mathrm{max}}} \subset (M_\C,J)$, we see that \(H = JT\partial M_\tau \cap T\partial M_\tau \) is a real \(J\)-invariant hyperplane bundle. The restriction $\alpha|_{\pmt}$ of the 1-form in \eqref{eqn:EQUAL} is a contact form on $\pmt$ with $\ker \alpha = H$. 

 The characteristic vector field given by \(T = \Xi_{\sqrt{\rho}}\) is the unique one on \(\pmt\) satisfying \(\alpha(T) = 1\) and \( d \alpha(T, \, \cdot\, ) = 0\). The complexification of the decomposition $T\partial M_\tau = H \oplus \R T$ yields a CR structure:
 \begin{align}
    T^{\C} \pmt = H^{(1,0)} \pmt \oplus H^{(0,1)} \pmt \oplus \C T,
\end{align}
where \(H^{(1,0)}\) and \(H^{(0,1)}\) are the $J$-holomorphic and $J$-antiholomorphic subspaces, respectively.

\subsection{The Szeg\H{o} projector and the Boutet de Movel--Sj\"{o}strand parametrix}

     The Szeg\H{o} projector \(\Pi_{\tau}\) associated to the boundary of a Grauert tube is the orthogonal projection
     \begin{align}
         \Pi_\tau \colon L^{2}(\pmt, d \mu_{\tau}) \to H^{2}( \pmt, d \mu_{\tau})
     \end{align}
     onto the Hardy space consisting of boundary values of holomorphic functions in $M_\tau$ that are square integrable with respect to the volume form \eqref{eqn:CONTACTVOLUME}. This  is a Fourier integral operator with a positive complex canonical relation whose real points are the graph of the identity map on the symplectic cone
\begin{equation}\label{eqn:SIGMATAU}
\Sigma_\tau = \set{(\zeta, r \alpha_\zeta): r \in \R_+} \subset T^*(\partial M_\tau)
\end{equation}
spanned by the contact form \eqref{eqn:EQUAL}.
		Using \eqref{eqn:DIFFEO}, we can construct a symplectic equivalence
\begin{equation}\label{eqn:IOTA}
\iota_\tau \colon T^*M - 0 \to \Sigma_\tau, \quad \iota_\tau(x, \xi) = \p[\bigg]{E\p[\Big]{x, \tau \frac{\xi}{\abs{\xi}}}, \abs{\xi} \alpha_{E(x, \tau \frac{\xi}{\abs{\xi}})}}.
\end{equation}

Details of the symbol of the Szeg\H{o} projector can be found in \cite[Theorem~11.2]{BoutetGuillemin81}. Briefly, let $\Sigma_\tau^\perp \otimes \C$ be the complexified normal bundle of \eqref{eqn:SIGMATAU}. The symbol $\sigma(\Pi_\tau)$ of $\Pi_\tau$ is a rank one projection onto a ground state $e_{\Lambda_\tau}$, which is annihilated by a Lagrangian system of Cauchy--Riemann equations corresponding to a Lagrangian subspace $\Lambda_\tau \subset \Sigma_\tau^\perp \otimes \C$.

The time evolution $\Pi_\tau \mapsto \gtc{-t}\Pi_\tau\gtc{t}$ under the Hamilton flow \eqref{eqn:GTTAU} yields another a rank one projection onto some time-dependent ground state $e_{\Lambda_\tau^t}$, where $\Lambda_\tau^t$ is the pushforward of $\Lambda_\tau$ under the flow. The quantity
\begin{equation}\label{eqn:OVERLAP}
\sigma_{t, \tau, 0} = \ang{e_{\Lambda_\tau^t} , e_{\Lambda_\tau}}^{-1}
\end{equation}
appears in \eqref{Dynamical Toeplitz} and \eqref{Dynamical Toeplitz Wave group}. See \cite[Section~4.3]{Zelditch14intersect} for details.

		The Szeg\H{o} kernel \(\Pi_\tau(z,w)\) is defined by the relation
     \begin{align}\label{eqn:SZEGOKERNEL}
         \Pi_\tau f(z) = \int_{\pmt} \Pi_\tau(z,w)f(w) \, d \mu_{\tau}(w) \quad  \text{for all \(f \in L^{2}(\pmt)\)}.
     \end{align}
		To describe an oscillatory integral representation for the kernel, we introduce the defining function
\begin{equation}\label{eqn:DEFININGFUNCTION}
\ptau \colon M_{\tau_\mathrm{max}} \to [0,\infty), \quad \ptau(z) = \rho(z) - \tau^2,
\end{equation}
so that $\ptau < 0$ in $M_\tau$ and $\ptau = 0$ on $\partial M_\tau$.
Let $\ptau(z,\overline{w})$ be the analytic extension of $\ptau(z) = \ptau(z,\overline{z})$ to $M_\tau \times \overline{M}_\tau$ obtained by polarization. 
\begin{align}\label{eqn:BSJPHASE}
    \psi_\tau(z,w) = \frac{1}{i} \ptau(z, \overline{w}) = \frac{1}{i}\left( - \frac{1}{4}r^{2}_{\C}(z, \overline{w} ) - \tau^{2} \right).
\end{align}
By construction, $\psi$ is holomorphic in $z$, antiholomorphic in $w$, and satisfies $\psi(z,w) = -\overline{\psi(z,w)}$. It appears as the phase function of the following parametrix for \(\Pi_{\tau}\) due to Boutet de Monvel and Sj\"{o}strand.

\begin{Th}[The Boutet de Monvel--Sj\"{o}strand parametrix, {\cite[Theorem~1.5]{BoutetSjostrand76}}]
     With $\psi$ as in \eqref{eqn:BSJPHASE}, there exists a classical symbol
\begin{equation}
s \in S^{m - 1 }( \pmt \times \pmt \times \R^{ +} ) \quad \text{with} \quad s(z,w, \sigma) \sim \sum_{k = 0}^{\infty}\sigma^{m - 1 - k}s_k(z,w) \label{Symbol Expansion BDM}
\end{equation}
so that the Szeg\H{o} kernel \eqref{eqn:SZEGOKERNEL} has the oscillatory integral representation
    \begin{align}
        \Pi_\tau(z,w) = \int_{0}^{\infty}e^{i \sigma \psi_\tau(z,w)}s(z, w,\sigma) \, d \sigma  \quad \text{modulo a smoothing kernel}. \label{Parametrix for Szego Kernel}
    \end{align}
\end{Th}

A key estimate for $\ptau$ (or equivalently for the phase function $\psi$) can be stated in terms of the \emph{Calabi diastatis function}, which is defined by
\begin{equation}\label{eqn:DIASTASIS}
D(z, w) = \ptau(z,\overline{z}) + \ptau(w, \overline{w}) - \ptau(z, \overline{w}) - \ptau(w, \overline{z}).
\end{equation}
In the closure of the Grauert tube, \cite[Corollary~1.3]{BoutetSjostrand76} gives the lower bound
\begin{align}\label{Diastasis inequality}
    D(z,w) \geq C \p[\big]{ d( z, \pmt ) + d( w, \pmt) + d(z,w)^{2}} \quad \text{for $z,w \in \overline{M_{\tau}}$}.
\end{align}

\subsection{The Toeplitz operator \texorpdfstring{$\Pi_\tau D_{\sqrt{\rho}}\Pi_\tau$}{PDP} and tempered eigenfunction sums}\label{sec:OPERATORS}
In this section, we introduce the two operators for which we compute the scaling asymptotics in \autoref{Graph Scaling Theorem} and \autoref{Graph Scaling Theorem2}.

The operator $\pdp$ is a generalized Toeplitz operator in the sense of Boutet de Monvel--Guillemin \cite{BoutetGuillemin81}. Here, \(D_{\sqrt{\rho} } = \frac{1}{i} \Xi_{\sqrt{\rho} }\) is differentiation along the Hamilton vector field of the Grauert tube function. The symbol of \(D_{\sqrt{\rho} }\) is nowhere vanishing on $\Sigma_\tau - 0$, where $\Sigma_\tau$ is the symplectic cone \eqref{eqn:SIGMATAU}. Thus, \(\Pi_\tau D_{\sqrt{\rho}}\Pi_\tau\) is elliptic and its spectrum discrete.

As discussed in \autoref{sec:FOURIERANALOGUE}, the Grauert tube analogue of Fourier coefficients of the Szeg\H{o} kernel are given by the spectral localizations
\begin{align}
    \pcl = \Pi_\tau \chi ( \Pi_\tau D_{\sqrt{\rho}} \Pi_\tau - \lambda )  = \int_{\R} \hat{\chi} (t) e^{ - it \lambda}\Pi_\tau e^{i t \Pi_\tau D_{\sqrt{\rho}}\Pi_\tau}  \, dt.
\end{align}
Here, $\hat{\chi}$ is a Schwartz function whose Fourier transform is compactly supported in some small neighborhood $[-\epsilon, \epsilon]$ of the origin, and $\widehat{\chi}(0) = 1$.

If we	denote by the pullback by the Hamilton flow \eqref{eqn:GTTAU} of $\Xi_{\sqrt{\rho}}$ on $\partial M_\tau$, then due to \cite[Proposition~5.3]{Zelditch20Husimi}, there exists a classical polyhomogeneous pseudodifferential operator \(\hat{\sigma}_{t, \tau}( w, D_{\sqrt{\rho} })\) on \(\pmt\) 
    so that
    \begin{align}\label{Dynamical Toeplitz}
        \Pi_\tau e^{i t \pdp } \sim \Pi_\tau\hat{\sigma}_{t, \tau} ( \gtc{t} )^{*} \Pi_\tau  \quad \text{modulo a smoothing Toeplitz operator}. 
    \end{align}
The symbol $\sigma_{t,\tau}$ of $\hat{\sigma}_{t,\tau}$ admits a complete asymptotic expansion
    \begin{align}\label{eqn:PDPTOEPLITZ}
        \sigma_{t, \tau}(w,r) \sim \sum_{j = 0}^{\infty}\sigma_{t, \tau, j}(w)r^{- j},
    \end{align}
in which $\sigma_{t,\tau, 0} =  \ang{e_{\Lambda_\tau^t} , e_{\Lambda_\tau}}^{-1}$ is to leading order the reciprocal of the overlap of two Gaussians, as in \eqref{eqn:OVERLAP}.

\begin{Rem}\label{rem:PCL}
It follows that
\begin{align}
\pcl(x,y) &\sim \p[\bigg]{\int_{\R} \hat{\chi} (t) e^{ - it \lambda} \Pi_\tau\hat{\sigma}_{t, \tau} ( \gtc{t} )^{*} \Pi_\tau \, dt}(x,y) \\
&= \int_\R\!\int_{\pmt} \widehat{\chi}(t) e^{-it\lambda}\Pi_\tau(x,w) \sigma_{t,\tau}(w) \Pi_\tau(\gtc{t}(w),y)\,dtdy.
\end{align}
In the proof of \autoref{Graph Scaling Theorem}, we replace the two Szeg\H{o} kernels in the expression above by the parametrices \eqref{Parametrix for Szego Kernel} and directly compute the resulting oscillatory integral in parameter $\sqrt{\lambda}$ using stationary phase.
\end{Rem}

We now introduce the tempered spectral projections kernel $\Pcl(z,w)$ constructed using analytically continued eigenfunctions. Recall the eigenequation \eqref{eqn:EIGENEQUATION} for the Laplacian on $M$. The eigenfunction expansion of the Schwartz kernel of the half-wave operator $U(t) = e^{it\sqrt{-\Delta}}$ is given by
\begin{align}\label{eqn:WAVE}
     U(t,x,y) = \sum_{j} e^{i t \lambda_j} \phi_{\lambda_j}(x) \overline{\phi_{\lambda_j}(y)}.
\end{align}

As shown in \cite{Boutet78, GuilleminStenzel92,GolseLeichtnamStenzel96,Zelditch12potential}, for $0 < \tau \le \tau_{\mathrm{max}}$, the Schwartz kernel $U(t,x,y)$ admits an analytic extension $U(t + i \tau,x,y)$ in the time variable $t \mapsto t + i\tau \in \C$, and then in the spacial variable $x \mapsto z \in M_\tau$. Let $\ocal^s(\partial M_\tau)$ denote the order $s$ Sobolev space of CR holomorphic functions on the Grauert tube boundary. Then the Poisson operator
\begin{equation}\label{eqn:POISSON}
U(i \tau) = e^{-\tau \sqrt{-\Delta}} \colon L^2(M) \to \ocal^{\frac{m-1}{4}}(\partial M_\tau)
\end{equation}
with kernel  $U(i\tau, z, y)$ is a Fourier integral operator of order $-(m-1)/4$ with complex phase associated to the canonical relation $\set{(y, \eta, \iota_\tau(y, \eta)} \subset T^*M \times \Sigma_\tau$. Here, the quantities $\iota_\tau$ and $\Sigma_\tau$ are defined in \eqref{eqn:IOTA} and \eqref{eqn:SIGMATAU}. We recall the following lemma.

\begin{Lemma}[{\cite[Lemma~8.2]{Zelditch12potential}}] \label{Unitary pseudo}
Let \(\Psi^{s}\) denote the class of psuedodifferential operators of order \(s\). Then,
\begin{enumerate}
    \item \(U(i \tau)^{*}U(i \tau) \in \Psi^{ -\frac{m - 1}{2}}\left( M \right)\) with principal symbol \(|\xi|_{g}^{ - \frac{m - 1}{2}}\).
    \item \(U(i \tau)U\left( i \tau \right)^{*} = \Pi_\tau A_{\tau} \Pi_{\tau}\) where \(A_{\tau} \in \Psi^{\frac{m - 1}{2}}\left( \pmt \right)\) has principal symbol \(|\sigma|_{g}^{\frac{m - 1}{2}}\) as a function on \(\Sigma_{\tau}\).
\end{enumerate}
\end{Lemma}

To introduce the complexified spectral projection kernels, we need to further continue $U(i\tau, z,y)$ anti-holomorphically in the $y$ variable. Consider the operator
\begin{equation}
U^\C(t + 2i \tau) = U(i\tau) U(t) U(i\tau)^* \colon \ocal(\partial M_\tau) \to \ocal(\partial M_\tau)
\end{equation}
with Schwartz kernel
\begin{equation}
U^\C(t + 2i\tau, z, w) = \sum_{j=1}^\infty e^{i ( t + 2 i \tau ) \mu_j}\phi^{\C}_{\lambda_j}(z) \overline{\phi^{\C}_{\lambda_j}(w)}.
\end{equation}
Set $t = 0$, then the partial sums of the expression above becomes
\begin{equation}\label{eqn:TEMPEREDPROJ}
P_\lambda(z,w) = \sum_{j: \lambda_j \le \lambda} e^{-2\tau \lambda_j}\phi^{\C}_{\lambda_j}(z) \overline{\phi^{\C}_{\lambda_j}(w)}.
\end{equation}

To smooth out the kernel, we fix \(\epsilon > 0\) and fix \(\chi\) a positive even Schwartz function such that \(\hat{\chi} \left( 0 \right) = 1 \) and \( \supp\hat{ \chi} \subset [ - \epsilon, \epsilon]\). Define
\begin{align}\label{eqn:TEMPEREDPCL}
    \Pcl(z,w) &= \chi \ast d_\lambda P_\lambda(z,w) \sim  \int_{\R}\hat{\chi} (t) e^{ - it \lambda} U^{\C}( t + 2 i \tau, z , w) \, dt.
\end{align}
As before, if we denote by $(\gtc{t})^*$ the pullback by the Hamilton flow \eqref{eqn:GTTAU} of $\Xi_{\sqrt{\rho}}$ on $\partial M_\tau$, then \cite[Proposition~7.1]{Zelditch20Husimi} establishes the existence a classical polyhomogeneous pseudodifferential operator \(\hat{\sigma}_{t, \tau}( w, D_{\sqrt{\rho} })\) on \(\pmt\) 
    so that
    \begin{align}\label{Dynamical Toeplitz Wave group}
        U^{\C}( t + 2 i \tau ) \sim \Pi_\tau \hat{\sigma}_{t, \tau} \left( \gtc{t} \right)^{*} \Pi_\tau \quad \text{modulo a smoothing Toeplitz operator}. 
    \end{align}
The symbol $\sigma_{t,\tau}$ of $\hat{\sigma}_{t,\tau}$ admits a complete asymptotic expansion
    \begin{align}\label{eqn:WAVEGROUPTOEPLITZ}
        \sigma_{t, \tau}\left( w,r \right) \sim \sum_{j = 0}^{\infty}\sigma_{t, \tau, j}\left( w \right)r^{ - \frac{m - 1}{2} - j},
    \end{align}
in which $\sigma_{t,\tau, 0} =  \ang{e_{\Lambda_\tau^t} , e_{\Lambda_\tau}}^{-1}$ is to leading order the reciprocal of the overlap of two Gaussians, as in \eqref{eqn:OVERLAP}.

\subsection{Quantization of linear symplectic maps on Bargmann--Fock space} \label{sec:METAPLECTIC}
The proofs of our theorems involve Taylor expansions in appropriate coordinates to reduce the geometry to the model linear space, so we briefly review the metaplectic representation on Bargmann--Fock space used to quantize symplectic linear mappings. Details can be found  \cite{Daubechies80,Folland89harmonic}

The Bargmann--Fock space on \(\C^m\) is
\begin{align}
    \mathcal{H}(\C^m) = \set[\Big]{f (z) e^{- \frac{|z|^2}{2}} \in L^2(\C^m,dz) \mathrel{}\Big\vert\mathrel{} f \in \ocal(\C^m)}.
\end{align}
The reproducing Bergman kernel has the exact formula
\begin{align}
    \Pi_{\mathcal{H}}(z,{w}) = \left( 2 \pi \right)^{-m} e^{ -\frac{|z|^2}{2} - \frac{|w|^2}{2} + z\bar{w}}.
\end{align}
Let \(Sp(m,\R)\) denote the space of real symplectic matrices on \(\R^{2m} = \R^m_x \times \R^m_y\) with respect to the standard symplectic form. Then matrix multiplication \(M \in Sp(m, \R)\) in real coordinates takes the form
\begin{align}\label{eqn:MATRIX}
    M \begin{pmatrix}
    x \\ y
    \end{pmatrix} =  \begin{pmatrix}A & B \\ C & D \end{pmatrix}\begin{pmatrix}
    x \\
    y
    \end{pmatrix} = \begin{pmatrix}
    x'\\
    y'
    \end{pmatrix}.
\end{align}

We map $\R^{2m}$ into $\C^{2m}$ via \((x,y) \mapsto (x + i y, x - iy)  =: (z, \bar{z})\). Under this mapping, \eqref{eqn:MATRIX} becomes
\begin{align}\label{Symplectic complexification}
    \mathcal{M}\begin{pmatrix}
    z  \\
    \bar{z}
    \end{pmatrix} =    \begin{pmatrix}
    P & {Q} \\
    \bar{Q} & \bar{P}
    \end{pmatrix} \begin{pmatrix}
    z \\
    \bar{z}
    \end{pmatrix} =  \begin{pmatrix}
    z' \\
    \bar{z}'
    \end{pmatrix},
\end{align}
where the holomorphic component $P$ and antiholomorphic component $Q$ of the symplectic mapping are given by
\begin{align}
    \begin{pmatrix}
    P & Q \\
    \bar{Q} & \bar{P}
    \end{pmatrix} = \mathcal{W}^{-1} \begin{pmatrix}
    A & B \\
    C& D
    \end{pmatrix} \mathcal{W}, \quad \mathcal{W} = \frac{1}{\sqrt{2}} \begin{pmatrix}
    I & I \\
    -iI & iI
    \end{pmatrix}.
\end{align}
(The choice of normalization is taken so that \(\mathcal{W}\) is unitary.) The explicit formula for the holomorphic component is
\begin{align}
    P = \frac{1}{2}(A + D + i (C - B)).
\end{align}

The metaplectic representation on \(\mathcal{H}(\C^m)\) is defined by $M \mapsto \Pi_{\mathcal{H},{M}}$, the latter being a unitary operator with kernel
\begin{align}
    \Pi_{\mathcal{H},{M}}(z,w) = (\det P )^{- \frac{1}{2}} \int_{\C^m}\Pi_{\mathcal{H}}(z, \mathcal{M}v) \Pi_{\mathcal{H}}(v, w)\,dv, \label{Gaussian Integral}
\end{align}
in which we set \(\mathcal{M} v : = Pv  + Q \bar{v} \). (The ambiguity of the sign of \(\left(\det P\right) ^{- \frac{1}{2}}\) is determined by the lift to the double cover.) Explicit computations involving standard Gaussian integrals show
\begin{align}
\Pi_{\mathcal{H}, {M}}(z,w) &:= \mathcal{K_M}(z, w) e^{- \frac{|z|^2}{2} - \frac{|w|^2}{2}},\\
\mathcal{K_M}(z,w) &:= ( 2 \pi )^{-m} (\det P)^{- \frac{1}{2}} \exp \left\{ \frac{1}{2}\left(z \bar{Q}P^{-1}z + 2\bar{w}P^{-1}z - \bar{w}P^{-1} Q \bar{w} \right)\right\} .
\end{align}

The principal term of \autoref{Graph Scaling Theorem} contains the lift of \(\Pi_{\mathcal{H},M}\) to  the reduced Heisenberg group \(\mathbb{H}^{m}_{\text{red}} \cong S^1 \times \C^{m} \), which is given by
\begin{align}
    \hat{\Pi}_{\mathcal{H}, M}(\theta,z, \varphi, w)  = e^{i(\theta - \varphi)}\Pi_{\mathcal{H},M}(z,w).
\end{align}

\section{Proofs of near graph scaling asymptotics}\label{sec:PROOFSCALING}

This section is focused on proving the scaling asymptotics \autoref{Graph Scaling Theorem} and \autoref{Graph Scaling Theorem2}, with the two proofs being identical. The techniques are similar to those of \cite{ChangRabinowitz21}.

\subsection{Identities in Heisenberg coordinates}
We briefly recall the notion of Heisenberg coordinates. We point the reader to \cite{FollandStein} for detailed construction of these coordinates on any strongly pseudoconvex CR manifold,   and to \cite{ChangRabinowitz21} for the Grauert tube setting. Roughly speaking, in Heisenberg coordinates, the strongly pseudoconvex boundary $\partial M_\tau \subset M_\C$ is, to a first approximation, the Heisenberg group viewed as a Seigel domain in complex Euclidean space. More precisely, given a point \(p \in \pmt \subset M_\C\) one may use the Levi procedure as in \cite[Section~18]{FollandStein} to construct holomorphic coordinates \((z_0,z_1, \dotsc , z_{n-1}) = (z_0, z')\) on an open neighborhood \(U \subset M_\C\) such that for \(w \in U \) 
\begin{align}
    \rho(z_0,z') = - \Imp z_0 + |z'|^2 + O\left(|z_0||z'| + |z'|^3\right) \label{Heisenberg Kahler Potential}
\end{align}

We note that  $(t,z') := (\Rep z_0, z')$ constrained by $\rho(z_0,z') = 0$ provides a coordinate system on the open neighborhood \(V = U \cap \pmt\) in \(\pmt\). We will refer to both the coordinates on \(M_\C\) as well as the coordinates on \(\pmt\) as Heisenberg coordinates. Furthermore, let \(Z_0 = \frac{T}{|T|}\) where \(T\) is the characteristic vector field and \(Z_1, ...,Z_{m-1}\) denote an orthonormal frame of \(T^{1,0}\pmt\). Then in Heisenberg coordinates centered at \(p\) we have 
 \begin{align}
     Z_0 \big|_p = \frac{\partial}{\partial t} \bigg|_p, \ Z_j \big|_p = \frac{\partial}{\partial z_j} \bigg|_p \label{Vector Fields at p}
 \end{align}
  We now record several Taylor expansions in Heisenberg coordinates established by the authors in \cite{ChangRabinowitz21} that will be useful in subsequent sections. In the following \(\lambda \in \R^{ +}\) is a parameter tending to \(\infty\). We state all of the identities in rescaled form as they appear in the main argument.

\begin{Lemma}[Expansion of the rescaled phase function]\label{Lemma 4.2} Let $\psi_\tau$ be as in \eqref{eqn:BSJPHASE}. In Heisenberg coordinates centered at \(p \in \pmt\) we have
\begin{align}\label{lambdapsitau full expansion}
   i\lambda \psi_{\tau}\bigg( \bigg(\frac{\theta}{\lambda}, \frac{u}{\sqrt{\lambda} }\bigg), \frac{w}{\sqrt{\lambda}} \bigg)  = - \sqrt{\lambda} \frac{i}{2} \Rep w_0 + \tilde{R}\bigg(\frac{\theta}{\lambda}, \frac{\Rep w_0}{\sqrt{\lambda} },\frac{u}{\sqrt{\lambda} },\frac{w'}{\sqrt{\lambda} }\bigg),
\end{align}
where 
\begin{align}
    \tilde{R} = \frac{i}{2}\theta - \frac{\abs{u}^{2}}{2} - \frac{\abs{w}^{2}}{2} + u \cdot \overline{w} + \lambda Q\bigg( \frac{\theta}{\lambda}, \frac{u}{\sqrt{\lambda} } , \frac{\Rep w_0}{\sqrt{\lambda} }, \frac{w'}{\sqrt{\lambda} }  \bigg) \label{Full R_1 remainder}
\end{align}
and \(Q\) takes the form 
\begin{align}
    Q\bigg( \frac{\theta}{\lambda}, \frac{u}{\sqrt{\lambda} } , \frac{\Rep w_0}{\sqrt{\lambda} }, \frac{w'}{\sqrt{\lambda} }  \bigg) & = O\left( \frac{\abs{\Rep w_0}|u|}{\lambda} + \frac{\abs{\Rep w_0}|w'|}{\lambda} \right) + O\left( \lambda^{ - \frac{3}{2}} \right)
\end{align}

\end{Lemma}
\begin{proof}
This is a special case of the computation immediately following \cite[Remark~4.6]{ChangRabinowitz21}. Briefly, let $\phi_\tau = i\psi_\tau$ be the defining function \eqref{eqn:DEFININGFUNCTION} obtained by polarizing \eqref{Heisenberg Kahler Potential}. Then, by \cite[Lemma~3.4]{ChangRabinowitz21}, in Heisenberg coordinates we may write
    \begin{align}
   i  \psi_{\tau}(z, \overline{w} ) = \phi_{\tau}(z, \overline{w} ) = \frac{i}{2}( z_0 - \overline{w}_0  ) + \sum_{j = 1}^{m - 1} z_j \overline{w}_j + R( z ,\overline{w}).
\end{align}

Here, the remainder term $R$ may be written as
\begin{align}
    R( z, \overline{w} ) =R_2(z_0, \overline{w}_0 , z', \overline{w}' ) +R_2(z_0, \overline{w}_0) +  R^{}_3(z', \overline{w}' ),
\end{align}
where  \(R_{2}( z_0, \overline{w}_0, z', \overline{w}'  )\) only contains terms of the form $z_0^{\alpha} \overline{w}'^{\beta}$ and $\overline{w_0}^{\alpha} z'^{\beta}$ with \(|\alpha| + |\beta| \geq 2\). Similarly, \(R_2 ( z_0, \overline{w}_0  )\) (resp.~\(R_{3}( z', \overline{w}'  )\)) only contains terms of the form \(z_0^{\alpha} \overline{w}_0^{\beta} \)  with \(|\alpha| + |\beta| \geq 2\) (resp.~terms of the form \(z'^{\alpha} \overline{w'}^{\beta} \) with \( |\alpha| + |\beta| \geq 3\)).

Keeping track of the powers of $\sqrt{\lambda}$ under parabolic rescaling results in the statement of the lemma.
\end{proof}

To prove our scaling asymptotics we will simultaneously be working with two sets of Heisenberg coordinate systems, one centered at \(p\) and another centered at \(\gtc{s}(p)\). We recall that \(\gtc{s}\) is the Hamiltonian flow of the characteristic vector field  which also preserves \(T^{1,0} \pmt  \oplus T^{0,1}\pmt \). Its derivative \(D\gtc{s}\) is a linear map \(T_p \pmt \to T_{\gtc{s}(p)} \pmt \). With respect to Heisenberg coordinates \eqref{Vector Fields at p} at $p$ and $\gtc{s}(p)$, we have
\begin{align}
    D\gtc{s} & = \begin{pmatrix} 1 & 0\\
   0 & M_s
 \end{pmatrix} \quad \text{with} \quad M_s \in Sp(m-1,\R).
\end{align}
We denote its complexification by \(\mathcal{M}_s\) as in \eqref{Symplectic complexification} and use the same notation $P,Q$ for its holomorphic and anti-holomorphic components. We have the following identities for Heisenberg coordinates centered at \(\gtc{s}(p)\).

\begin{Lemma}[Expansion of the rescaled geodesic flow]\label{lem:TAYLORFLOW}
Let \(w = (w_0,w')\) be a point in a Heisenberg coordinate chart centered at \(p \in \pmt\). Then, in Heisenberg coordinates centered at \(\gtc{s}(p) \in \pmt \), we have  
\begin{multline}
    \gtc{s + \frac{r}{\sqrt{\lambda}}}\left(\frac{w}{\sqrt{\lambda}}\right) = \left( \frac{\Rep w_0}{\sqrt{\lambda}}+ \frac{2 \tau r}{\sqrt{\lambda}} +  O\left(\frac{|r|}{\lambda}\right) + O\left(\lambda^{-3/2}\right)\right. ,\\
    \left.\frac{\mathcal{M}_s w'}{\sqrt{\lambda}}  + O \left( \frac{|r|}{\lambda}\right) + O\left(\lambda^{-3/2} \right)\right).
\end{multline}
\end{Lemma}
\begin{proof}
  This follows from the Taylor expansion $ \gtc{t}(z_0,z') = \p{ z_0 + 2 \tau t + t \cdot O^1 + O( t^{2} ), 	z' + t \cdot O^1 + O( t^{2})}$ of \cite[Lemma~3.6]{ChangRabinowitz21}.
\end{proof}

\begin{Lemma}[Combined expansion of the rescaled phase and flow]
In Heisenberg coordinates centered at \( \gtc{s}(p) \in \pmt\) we have
\begin{equation}\label{lambdapsitau full expansion 1}
\begin{aligned}
   i \lambda \psi_{\tau}\bigg(\gtc{s + \frac{r}{\sqrt{\lambda}}}\left( \frac{w}{\sqrt{\lambda}}\right),\bigg( \frac{\phi}{\lambda}, \frac{v}{\sqrt{\lambda} } \bigg)\bigg) &= \sqrt{\lambda} \frac{i}{2}\left( \Rep w_0 + 2 \tau r \right) \\
	&\quad + \tilde{S}\bigg( \frac{\varphi}{\lambda},  \frac{\Rep w_0}{\sqrt{\lambda} }, \frac{r}{\sqrt{\lambda} }, \frac{v}{\sqrt{\lambda} }, \frac{w'}{\sqrt{\lambda} } \bigg),
 \end{aligned}
\end{equation}
 where
 \begin{align}
     \tilde{S} = - \frac{i}{2}\varphi - \frac{\abs{v}^{2}}{2} - \frac{ \abs{\mathcal{M}_s  w}^{2}}{2} + \overline{v} \cdot (\mathcal{M}_sw) + \lambda  T\bigg( \frac{\varphi}{\lambda}, \frac{v}{\sqrt{\lambda} },   \frac{r}{\sqrt{\lambda} }, \frac{\Rep w_0}{\sqrt{\lambda} }, \frac{w'}{\sqrt{\lambda} } \bigg) \label{Full S_1 remainder}
 \end{align}
and \(T\) takes the form
\begin{align}
      T\bigg( \frac{\varphi}{\lambda}, \frac{v}{\sqrt{\lambda} },   \frac{r}{\sqrt{\lambda} }, \frac{\Rep w_0}{\sqrt{\lambda} }, \frac{w'}{\sqrt{\lambda} } \bigg)  & = O\left( \frac{\abs{r}\abs{v}}{\lambda} + \frac{\abs{r}\abs{w'}}{\lambda} + \frac{\abs{\Rep w_0}\abs{v}}{\lambda} + \frac{\abs{\Rep w_0} \abs{w'}}{\lambda} \right) + O\left( \lambda^{ - \frac{3}{2}} \right)
\end{align}
\end{Lemma}

\begin{proof}
This follows from the\autoref{Lemma 4.2} and \autoref{lem:TAYLORFLOW}.
\end{proof}

\subsection{Proof of \autoref{Graph Scaling Theorem}: asymptotic expansion for \texorpdfstring{$\pcl$}{}}
 Fix \(p \in \pmt\) and let \(( \theta, u ), ( \varphi, v) \in \partial M_\tau\) be two points in Heisenberg coordinates centered at \(p\) and \(\gtc{s}(p)\) respectively. Then, as discussed in \autoref{rem:PCL}, substituting the parametrix \eqref{Parametrix for Szego Kernel} for each instance of $\Pi_\tau$ above and composing the resulting kernels, we arrive at the oscillatory integral representation
\begin{multline}\label{eqn:OSCINTORIGINAL}
    \Pi_{\chi, \lambda}\bigg( p + \left( \frac{\theta}{\lambda}, \frac{u}{\sqrt{\lambda} } \right), \gtc{s}(p) + \left(\frac{\varphi}{\lambda}, \frac{v}{\sqrt{\lambda} } \right) \bigg) \\
    \sim \int_{\R \times \partial M_\tau \times \R^+ \times \R^+} e^{i \lambda \Psi} A \,\quaddifferential,
\end{multline}
in which the phase $\Psi$ and the amplitude $A$ are given by
\begin{equation}\label{eqn:AMPA}
\begin{aligned}
\Psi &= -t + \frac{1}{\lambda} \sigma_2 \psi_\tau\bigg( p + \bigg( \frac{\theta}{\lambda}, \frac{u}{\sqrt{\lambda} } \bigg), w \bigg) + \frac{1}{\lambda} \sigma_1   \psi_{\tau}\bigg( \gtc{t}(w), \gtc{s}(p) + \bigg( \frac{\varphi}{\lambda}, \frac{v}{\sqrt{\lambda} } \bigg) \bigg),\\
A &= \hat{\chi}(t) s\bigg( p + \left( \tl, \ul \right), w, \sigma_2 \bigg)s\bigg( \gtc{t}(w), \gtc{t}(p) + \left(\phil, \vl\right), \sigma_1 \bigg) \sigma_{t, \tau}(w, \sigma_1  ).
\end{aligned}
\end{equation}
From now on, we suppress \(p \) and  \(\gtc{s}(p)\) from the notation, keeping in mind that they are the origin in each of their respective coordinates.
Make the change-of-variables \(\sigma_j \mapsto \lambda \sigma_j\). Homogeneity of the symbols implies
\begin{align}\label{Partial Szego Kernel Scaled}
    \Pi_{\chi, \lambda}\bigg( \frac{\theta}{\lambda}, \frac{u}{\sqrt{\lambda} }, \frac{\varphi}{\lambda}, \frac{v}{\sqrt{\lambda} } \bigg) \sim \lambda^{2m}\int_{\R \times \partial M_\tau \times \R^+ \times \R^+} e^{i \lambda \tilde{\Psi}} \tilde{A} \,\quaddifferential,
\end{align}
in which the phase $\tilde{\Psi}$ and the amplitude $\tilde{A}$ are given by
\begin{equation}\label{eqn:AMPANDPHASE}
\begin{aligned}
\tilde{\Psi} &= -t +  \sigma_2 \psi_\tau\bigg( \bigg( \frac{\theta}{\lambda}, \frac{u}{\sqrt{\lambda} } \bigg), w \bigg) + \sigma_1   \psi_{\tau}\bigg( \gtc{t}(w), \bigg( \frac{\varphi}{\lambda}, \frac{v}{\sqrt{\lambda} } \bigg) \bigg),\\
\tilde{A} &= \lambda^{-2m} A.
\end{aligned}
\end{equation}

We begin by localizing in \((w,t) \in \partial M_\tau \times \R\). Fix $C > 0$ and \(0 < \delta < \frac{1}{2}\). Set
\begin{equation}\label{eqn:CUTOFF}
\begin{aligned}
    V_{\lambda} & = \set*{(w,t) : \textstyle  \max\Big\{ d\big(w, \big(\frac{\theta}{\lambda}, \frac{u}{\sqrt{\lambda} }\big) \big), d\big( \gtc{t}(w), \big( \frac{\varphi}{\lambda}, \frac{v}{\sqrt{\lambda} } \big) \big) \Big\}< \frac{2C}{3}\lambda^{\delta - \frac{1}{2}}},\\
    W_{\lambda} & = \set*{(w,t)  : \textstyle \max \Big\{d\big(w, \big(\frac{\theta}{\lambda}, \frac{u}{\sqrt{\lambda} } \big)\big), d\big( \gtc{t}(w), \big( \frac{\varphi}{\lambda}, \frac{v}{\sqrt{\lambda} } \big) \big)\Big\}  > \frac{C}{2}\lambda^{\delta - \frac{1}{2}}}.
\end{aligned}
\end{equation}
Let \(\{\varrho_{\lambda}, 1 - \varrho_{\lambda}\} \) be a partition of unity subordinate to the cover $\{V_\lambda, W_\lambda\}$ and decompose the integral \eqref{Partial Szego Kernel Scaled} into
\begin{align}
     \Pi_{\chi, \lambda}\bigg( \frac{\theta}{\lambda}, \frac{u}{\sqrt{\lambda} }, \frac{\varphi}{\lambda}, \frac{v}{\sqrt{\lambda} } \bigg) & \sim I_1 + I_2,\\
		I_1 &= \lambda^{2m}\int e^{i \lambda \tilde{\Psi}}  \varrho_{\lambda}( t,w )  \tilde{A} \, \quaddifferential ,\\
		I_2 &= \lambda^{2m}\int e^{i \lambda \tilde{\Psi}} ( 1 - \varrho_{\lambda}( t,w ) ) \tilde{A} \, \quaddifferential.
\end{align}

\begin{Lemma}\label{Localization in time and space}
    We have $I_2 =  O( \lambda^{ - \infty})$.
\end{Lemma}

\begin{proof}
    By definition of $W_\lambda$, on the support of \(1 - \varrho_{\lambda}\) either
    \begin{align}
       \abs{d_{\sigma_2} \tilde{\Psi}} &=  \abs*{\psi_{\tau} \bigg(  \bigg( \frac{\theta}{\lambda}, \frac{u}{\sqrt{\lambda} } \bigg),w \bigg)} \geq 2 D\bigg(\bigg( \frac{\theta}{\lambda}, \frac{u}{\sqrt{\lambda} } \bigg), w \bigg) \geq C' \lambda^{2\delta - 1}
			\end{align}
			or
			\begin{align}
       \abs{d_{\sigma_1} \tilde{\Psi}} &=  \abs*{\psi_{\tau} \bigg(  \bigg( \gtc{t}(w),\frac{\phi}{\lambda}, \frac{v}{\sqrt{\lambda} } \bigg) \bigg)} \geq 2 D\bigg(\bigg( \gtc{t}(w), \frac{\phi}{\lambda}, \frac{v}{\sqrt{\lambda} } \bigg)\bigg) \geq C' \lambda^{2 \delta - 1}
     \end{align}
where \(D\) is the Calabi diastasis \eqref{eqn:DIASTASIS} and the inequalities follow from \eqref{Diastasis inequality}. Repeated integration by parts in \(\sigma_1\) or \(\sigma_2\) as appropriate completes the proof.
\end{proof}

In preparation for stationary phase we make the following change of variables 
\begin{align}\label{eqn:RESCALE}
t \mapsto s + \frac{r}{\sqrt{\lambda}} \quad \text{and} \quad \text{$(\Rep w_0, w') \mapsto \p[\bigg]{\frac{\Rep w_0}{\sqrt{\lambda}}, \frac{w'}{\sqrt{\lambda}}}$}.
\end{align}
Substituting in our formulas from lemmas 3.2 and 3.4 we obtain the following oscillatory integral with parameter \(\sqrt{\lambda}\). 

\begin{align}
    \pcl \left(\frac{\theta}{\lambda}, \frac{u}{\sqrt{\lambda}}, \frac{\phi}{\lambda}, \frac{v}{\sqrt{\lambda}}\right) \sim e^{- i s \lambda} \lambda^m \int e^{i \lambda \tilde{\Psi}} \tilde{A} d \sigma_1 d \sigma_2 dw dr \label{OSCIINT2}
\end{align} 
where
\begin{equation}\label{eqn:AMPPHASE2}
\begin{aligned}
\tilde{\Psi} &= - r - \frac{\sigma_2}{2}\Rep w_0 +\frac{\sigma_1}{2}(\Rep w_0 + 2 \tau r ), \\
\tilde{A} &= e^{\sigma_2\tilde{R} + \sigma_1 \tilde{S}}  \varrho_{\lambda} \tilde{A}\bigg( \frac{\theta}{\lambda} , \frac{u}{\sqrt{\lambda} },\frac{\varphi}{\lambda},  \frac{v}{\sqrt{\lambda} } ,  \frac{\Rep w_0}{\sqrt{\lambda} }, \frac{w'}{\sqrt{\lambda} } ,\frac{r}{\sqrt{\lambda} } ,    \sigma_1, \sigma_2  \bigg) J \bigg( \frac{w'}{\sqrt{\lambda}} \bigg),
\end{aligned}
\end{equation}
with $J(\,\cdot\,)$ the volume density in Heisenberg coordinates.

We may further localize this integral in the \(\sigma_1, \sigma_2\) variables. Let \(\set{\eta, 1 - \eta}\) be a partition of unity subordinate to the cover
\begin{equation}
\set*{(\sigma_1, \sigma_2) : 0 < \sigma_1, \sigma_2  < \frac{2}{\tau}} \quad \text{and} \quad \set*{(\sigma_1, \sigma_2) : \sigma_1, \sigma_2 > \frac{3}{2 \tau} }.
\end{equation}
Decompose \eqref{OSCIINT2} into two integrals:
\begin{align}
\Pi_{\chi, \lambda}\bigg( \frac{\theta}{\lambda}, \frac{u}{\sqrt{\lambda} }, \frac{\varphi}{\lambda}, \frac{v}{\sqrt{\lambda} } \bigg) &\sim I_1' + I_2',\\
I_1' &= e^{- i s \lambda}\lambda^{m}\int e^{i \lambda \tilde{\Psi}}  \eta(\sigma_1, \sigma_2) \tilde{A} \, d\sigma_1 d\sigma_2 d\mu_{\tau}(w) dr\\
I_2' &= e^{- i s \lambda}\lambda^{m}\int e^{i \lambda \tilde{\Psi}}  (1 - \eta(\sigma_1, \sigma_2))  \tilde{A} \, d\sigma_1 d\sigma_2 d\mu_{\tau}(w) dr
\end{align}
with $\tilde{A}$ and $\tilde{\Psi}$ as in \eqref{eqn:AMPPHASE2}. 
\begin{Lemma}\label{sigma localization}
We have $I_2' = O(\lambda^{-\infty})$.
\end{Lemma}
\begin{proof}
Notice that 
\begin{align}
    \abs*{ \nabla_{\Rep w_0,t} \tilde{\Psi}}^2 \geq \bigg( \frac{\sigma_1}{2} - \frac{\sigma_2}{2} \bigg) ^{2} + ( \tau\sigma_1 - 1 )^{2} \ge \frac{1}{4}
\end{align}
on the support of \(1 - \eta\). Thus, the lemma follows from repeated integration by parts in \((\Rep w_0, t)\).
\end{proof}

We have reduced the spectral localization kernel to the oscillatory integral
\begin{equation}\label{eqn:OSCINTFINAL}
\Pi_{\chi, \lambda}\bigg( \frac{\theta}{\lambda}, \frac{u}{\sqrt{\lambda} }, \frac{\varphi}{\lambda}, \frac{v}{\sqrt{\lambda} } \bigg) \sim e^{- i s \lambda} \lambda^{m} \int e^{i \sqrt{\lambda} \dtilde{\Psi}} \dtilde{A} \,dw' d(\Rep w_0) d\sigma_1 d\sigma_2 dr,
\end{equation}
with phase and amplitude
\begin{equation}
\begin{aligned}
\dtilde{\Psi} &= - r - \frac{\sigma_2}{2}\Rep w_0 +\frac{\sigma_1}{2}(\Rep w_0 + 2 \tau r ), \\
\dtilde{A} &=e^{\sigma_2\tilde{R} + \sigma_1 \tilde{S}}  \eta \varrho_{\lambda} \hat{\chi} \tilde{A} J.
\end{aligned}
\end{equation}
Since the exponential of the terms of order \(\lambda^{-\frac{1}{2}}\) appearing in \(\tilde{R}, \tilde{S}\) is bounded it may be absorbed into the main amplitude. We will now reduce \eqref{eqn:OSCINTFINAL} to a Gaussian integral over \(\C^{m-1}\) by integrating out the variables \(\Rep w_0, \sigma_1, \sigma_2, r\) using the method of stationary phase. We note the following derivatives:
\begin{align}
     \partial_{\sigma_2} \dtilde{\Psi} & = - \frac{1}{2} \Rep w_0, &   \partial_{\sigma_1} \dtilde{\Psi} &= \frac{1}{2}( \Rep w_0 + 2 \tau r ),\\
     \partial_{t} \dtilde{\Psi} & = - 1 + \tau{\sigma_1}, & \partial_{\Rep w_0} \dtilde{\Psi}  &= - \frac{\sigma_2}{2} + \frac{\sigma_1}{2}.
 \end{align}
The critical set of the phase is the point \(C =\{ \Rep w_0 = 0, r = 0, \sigma_1 = \sigma_2 = \frac{1}{\tau}\} \). The Hessian matrix and its inverse at the critical point are
\begin{align}
   \dtilde{\Psi}''_{C} &=  \p*{\begin{array}{c|cccc}  & r & \sigma_1 & \sigma_2 & \Rep w_0 \\ \hline  r  & 0 & \tau & 0 & 0\\ \sigma_1 & \tau & 0 & 0 & \frac{1}{2} \\ \sigma_2 & 0 & 0 & 0 & - \frac{1}{2} \\ \Rep w_0 & 0 &\frac{1}{2} & - \frac{1}{2} & 0\end{array}}, \quad (\dtilde{\Psi}''_C)^{-1} =      \begin{pmatrix} 0 & \frac{1}{\tau} & \frac{1}{\tau} & 0 \\
    \frac{1}{\tau} & 0 & 0 &0 \\ \frac{1}{\tau} & 0 &0 &- 2\\ 0 & 0 &- 2 & 0 \end{pmatrix} .
 \end{align}

Set
 \begin{align}
     L_{\dtilde{\Psi}} = \ang*{(\dtilde{\Psi}''_C)^{-1} D, D} =  \frac{2}{\tau}\partial_{\sigma_1}\partial_{r} + \frac{2}{\tau} \partial_{\sigma_2}\partial_r - 4 \partial_{\sigma_2}\partial_{\Rep w_0}
 \end{align}
By the method of stationary phase (\cite[Theorem 7.75]{Hormanderv1}), we have
 \begin{multline} 
    e^{- is \lambda} \lambda^m \int_{\R\times \R^{ + } \times \R^{ +} \times \R}e^{i \sqrt{\lambda} \dtilde{\Psi}} \dtilde{A} \, dw' d\sigma_1 d\sigma_2dr \\ 
		= \frac{8 \pi^{2}}{\lambda \tau} e^{- i s \lambda} \lambda^m \sum_{j = 0}^{N-1}\lambda^{ - \frac{j}{2}}\sum_{\nu - \mu = j} \sum_{2\nu \ge 3 \mu} \frac{1}{i^j 2^\nu} L_{\dtilde{\Psi}}^\nu \bra*{e^{\mu(\sigma_2\tilde{R} + \sigma_1 \tilde{S})}  \eta \varrho_{\lambda} \hat{\chi} \tilde{A} J }_{C} +\hat{R}_N,  \label{Abstract Stationary Phase Expansion}
 \end{multline}
with the remainder term satisfying
\begin{align}
\int_{\C^{m - 1}}^{} &\abs{\hat{R}_{N}}\,dw' \leq \lambda^{ - \frac{N}{2}} \int_{\C^{m - 1}}  \sum_{ \abs{\alpha} \leq 2 N}\sup \abs{D^{\alpha}(\eta \rho_\lambda \hat{\chi} \tilde{A} J)} \,dw' \le C_N \lambda^{-\frac{N}{2}}.
\end{align}
(Here, the supremum and the derivative $D^\alpha$ are taken over $t, \sigma_1, \sigma_2, \Rep w_0$ and the integral is with respect to the remaining variable $w'$. Note that $\tilde{A}$, defined in \eqref{eqn:AMPANDPHASE}, is a symbol of order zero.)

Thanks to the remainder estimate, we may integrate the asymptotic expansion \eqref{Abstract Stationary Phase Expansion} term-by-term in $w'$ to obtain \eqref{eqn:OSCINTFINAL}. Upon substituting expressions \eqref{Full R_1 remainder} and \eqref{Full S_1 remainder} the leading term is given by the following Gaussian integral

\begin{multline}
    C_m  \frac{\lambda^{m-1}}{\tau} e^{-is \lambda} \sigma_{s, \tau, 0}(p) e^{\frac{i}{2\tau}(\theta-\varphi)}\\
    \times \int_{\C^{m-1}}\exp\set[\bigg]{\frac{1}{\tau}\p[\bigg]{- \frac{|u|^2}{2}- \frac{|w'|^2}{2} + u \cdot \bar{w}' - \frac{|v|^2}{2} - \frac{|\mathcal{M}_sw'|^2}{2} + \bar{v} \cdot \mathcal{M}_sw' }}\,dw' \label{Gaussian Integral 2}
\end{multline}

The symbol \(\sigma_{s, \tau,0}(p)\) can be computed as in \cite{Zelditch20Husimi} to be \((\det P_s)^{- \frac{1}{2}}\). This is precisely the same integral as \eqref{Gaussian Integral} and so we obtain the leading term

\begin{align}
  \frac{C_m}{\tau} \left(\frac{\lambda}{\tau}\right)^{m-1}e^{- i s \lambda} \hat{\Pi}_{\mathcal{H}, M_s}\left(\frac{\theta}{2 \tau}, \frac{u}{\sqrt{\tau} }, \frac{\varphi}{2 \tau}, \frac{v}{\sqrt{\tau} }\right).
\end{align}
The lower order terms have the form
\begin{multline}
    {\frac{C_m}{\tau^m}}\lambda^{m-1 - \frac{j}{2}} e^{- i s \lambda} e^{\frac{i}{2\tau}(\theta-\varphi)} \\
    \times \int_{\C^{m-1}}P_j(u,v,w,s, \theta, \varphi)e^{\frac{1}{\tau}\p[\big]{- \frac{|u|^2}{2}- \frac{|w'|^2}{2} + u \cdot \bar{w}' - \frac{|v|^2}{2} - \frac{|\mathcal{M}_sw'|^2}{2} + \bar{v} \cdot \mathcal{M}_sw' }}\,dw',
\end{multline}
with \(j\) a positive integer and \(P_j(u,v,w,s,\theta,\varphi)\) a polynomial. This can be rewritten as
\begin{multline}
    {\frac{C_m}{\tau^m}}\lambda^{m-1 - \frac{j}{2}} e^{- i s \lambda} e^{\frac{i}{2\tau}(\theta-\varphi)}\\
    \times \int_{\C^{m-1}}e^{\frac{1}{\tau} \p[\big]{- \frac{|u|^2}{2}-\frac{|v|^2}{2} + u \cdot \bar{w}' + \bar{v} \cdot \mathcal{M}_s w' } }\hat{P}_j(u,v,s, \theta, \varphi,D)e^{\frac{1}{\tau}\p[\big]{- \frac{|w'|^2}{2}  -  \frac{|\mathcal{M}_s w'|^2}{2} }}\,dw',
\end{multline}
where \(\hat{P}_j\) is a differential operator with polynomial coefficients. We integrate by parts with the \(\hat{P}_j\) operator   from which we obtain the same Gaussian integral as \eqref{Gaussian Integral 2} against a polynomial independent of \(w'\). As a result, the lower order terms in the asymptotic expansion take the form
\begin{align}
  \frac{C_m}{\tau} \left(\frac{\lambda}{\tau}\right)^{m-1- \frac{j}{2}}e^{- i s \lambda}P_j(p,s,\tau,u,v,\theta,\varphi) \hat{\Pi}_{\mathcal{H}, M_s}\left(\frac{\theta}{2 \tau}, \frac{u}{\sqrt{\tau} }, \frac{\varphi}{2 \tau}, \frac{v}{\sqrt{\tau} }\right).
\end{align}

\subsection{Proof of \autoref{Graph Scaling Theorem2}: asymptotic expansion for \texorpdfstring{$\Pcl$}{}}
We can also study the on-shell scaling asymptotics for the tempered spectral projection kernel \eqref{eqn:TEMPEREDPCL} under Heisenberg-type rescaling. The proof is nearly identitical to that of \autoref{Graph Scaling Theorem}. We first write out the kernel using \autoref{Dynamical Toeplitz Wave group} and \eqref{Parametrix for Szego Kernel}:
\begin{align}\label{Partial Wave Group}
    \Pcl\bigg( \frac{\theta}{\lambda}, \frac{u}{\sqrt{\lambda} }, \frac{\varphi}{\lambda}, \frac{v}{\sqrt{\lambda} } \bigg) \sim \int_{\R \times \partial M_\tau \times \R^+ \times \R^+} e^{i \lambda \Psi} B \,\quaddifferential,
\end{align}
in which the phase $\Psi$ and the amplitude $B$ are given by
\begin{equation}
\begin{aligned}
\Psi &= -t + \frac{1}{\lambda} \sigma_2 \psi_\tau\bigg( \bigg( \frac{\theta}{\lambda}, \frac{u}{\sqrt{\lambda} } \bigg), w \bigg) + \frac{1}{\lambda} \sigma_1   \psi_{\tau}\bigg( \gtc{t}(w), \bigg( \frac{\varphi}{\lambda}, \frac{v}{\sqrt{\lambda} } \bigg) \bigg),\\
B &= \hat{\chi}(t) s\bigg( \tl, \ul, w, \sigma_2 \bigg)s\bigg( \gtc{t}(w), \phil, \vl, \sigma_1 \bigg) \sigma_{t, \tau}(w, \sigma_1  ).
\end{aligned}
\end{equation}

The only modification is that despite identical notation, the unitarization symbol $\sigma_{t,\tau}$ for $B$ is now of order $-(m-1)/2$, whereas $\sigma_{t,\tau}$ in the expression for $A$ in \eqref{eqn:AMPA} is of order zero. Hence, the oscillatory integral expression \eqref{Partial Wave Group} is exactly $\lambda^{-(m-1)/2}$ times the expression \eqref{eqn:OSCINTORIGINAL}. The rest of the computations proceed in the same manner.

\section{Proofs of \texorpdfstring{$L^p$}{Lp} estimates}\label{sec:LP}

In this section we prove $L^p$ estimates of the Szeg\H{o} kernel, namely \autoref{theo:MAINTHEO} and \autoref{Short window}. \autoref{cor:LP}, an $L^p$ upper bound for normalized eigenfunctions of $\pdp$, is then deduced.

We begin by establishing the following Gaussian decay estimate for $\pcl(z,w)$ away from a small neighborhood of the graph $(z,w) = (p, \gtc{s}(p))$.

\begin{Lemma}[Gaussian decay estimate]\label{GAUSSIAN}
 Fix \(z \in \pmt\). Set $\delta = \abs{\supp \hat{\chi}} = 2\epsilon$ and \(T_\delta(z) = \{\gtc{t}(z): |t| < \delta\}\). Then, after possibly shrinking \(\supp \hat{\chi}\), there exists \(C >0\) such that whenever \(d(T_\delta(z),w) < C\lambda^{-\frac{1}{3}}\)  we have
 \begin{align}
     \abs{\pcl(z,w)} \leq C(1 + o(1))\lambda^{m-1} e^{- \frac{1- \epsilon}{2\tau}\lambda d(T_\delta(z),w)^2} + O(\lambda^{-\infty}).
 \end{align}
\end{Lemma}

\begin{proof}
 Let $\abs{s} < \epsilon = \delta/2$. In Heisenberg coordinates centered at $\gtc{s}(z)$, consider points of the form \(w = \gtc{s}(z) + (\frac{\varphi}{\lambda}, \frac{u}{\sqrt{\lambda}})\) with \(|(\varphi, u)| < \lambda^{\frac{1}{6}}\). We repeat the stationary phase computation in the proof of \autoref{Graph Scaling Theorem}:
    \begin{align}\label{eqn:OSCINTFINAL2}
 \Pi_{\chi, \lambda}\bigg( z, \gtc{s}(z) + \bigg(\frac{\varphi}{\lambda}, \frac{v}{\sqrt{\lambda} }\bigg) \bigg) &\sim e^{- i s \lambda} \lambda^{m} \int e^{i \sqrt{\lambda} \dtilde{\Psi}} \dtilde{A} \,dw' d(\Rep w_0) d\sigma_1 d\sigma_2 dt,
\end{align}
with phase and amplitude defined in the same way as \eqref{eqn:OSCINTFINAL}. Keeping track of the first order remainder term, we find
\begin{align}
    \left| \Pi_{\chi, \lambda}\bigg( z, \gtc{s}(z) + \bigg(\frac{\varphi}{\lambda}, \frac{v}{\sqrt{\lambda} }\bigg) \bigg) \right| &= \lambda^{m-1}e^{- \frac{|u|^2}{2\tau} - \frac{\bar{u}P^{-1}_sQ_s  \bar{u}}{2\tau}}\\
    &\quad+ \lambda^{m-1-\frac{1}{2}}e^{- \frac{1- \epsilon}{2 \tau} \left( |u|^2 + \bar{u}P^{-1}_s Q_s \bar{u}\right)}R(p,s,u,\lambda),
 \end{align}
  where $P_s,Q_s$ are matrices defined in \eqref{Symplectic complexification} and $R(p,s,u,N) \leq C(s) |u|$ for $\abs{u} < \lambda^{\frac{1}{6}}$.

  Since \(P_s^{-1}Q_s = o(s)\)
 and \(|u| \approx \sqrt{\lambda}d (\gtc{s}(z), \gtc{s}(z) + \frac{u}{\sqrt{\lambda}})\), we get uniformly for \(|s| < \epsilon\) and $\abs{u} < \lambda^{\frac{1}{6}}$ that
  \begin{align}
     \left| \Pi_{\chi, \lambda}\bigg( z, \gtc{s}(z) + \bigg(\frac{\varphi}{\lambda}, \frac{v}{\sqrt{\lambda} }\bigg) \bigg) \right| \leq C\lambda^{m-1}\left(1 + \frac{|u|}{\sqrt{\lambda}}\right) e^{-\frac{1-\epsilon}{2}d \left(\gtc{s}(z), \gtc{s}(z) + \frac{u}{\sqrt{\lambda}} \right)},
 \end{align}
as desired.
\end{proof}

To establish sharpness we will also need the following lower bound on \(\pcl\) in a \(\lambda^{ - \frac{1}{2}} \) of the graph.

\begin{Lemma}\label{Lower Bound}
   Fix \(z \in \pmt\) and \(D > 0\). Set $\delta = \abs{\supp \hat{\chi}} = 2\epsilon$ and \(T_\delta(z)\) as in \autoref{GAUSSIAN}. Then, after possibly shrinking \(\supp \hat{\chi}\), there exists \(C >0\) such that whenever \(d(T_\delta(z),w) < D \lambda^{-\frac{1}{2}}\)  we have
 \begin{align}
     \abs{\pcl(z,w)} \geq C(1 - o(1))\lambda^{m-1} e^{- \frac{1 + \epsilon}{2\tau}\lambda d(T_\delta(z),w)^2}. 
 \end{align}
\end{Lemma}

\begin{proof}
  This is an immediate corollary of \autoref{Graph Scaling Theorem} when we take \(w = \gtc{s}(z) + (\frac{\varphi}{\lambda}, \frac{u}{\sqrt{\lambda}})\) with \(|(\varphi, u)| < D\).
\end{proof}

\subsection{Proof of \autoref{theo:MAINTHEO}: sharp norm estimates for \texorpdfstring{$\pcl$}{}}

  We invoke the Shur-Young inequality
  \begin{align}
    \norm{\pcl}_{L^{p} \to L^{q}} \leq C_p \left[ \sup_z \int_{\pmt} |\pcl(z,w)|^{r} \,d w \right] ^{\frac{1}{r}}, \quad \frac{1}{r} = 1 - \frac{1}{p} + \frac{1}{q}. \label{Shur - Young}
  \end{align}
With \(T_\delta\) as in \autoref{GAUSSIAN}, we break up the integral 
\begin{align}
\int_{\pmt} | \pcl(z,w)|^r\, dw &= \int_{d(T_{\delta}(z), w)\leq \lambda^{- \frac{1}{3}}} | \pcl(z,w)|^r \, dw \label{close to tube} \\
&\quad+ \int_{d(T_{\delta}(z), w)\geq \lambda^{- \frac{1}{3}}} | \pcl(z,w)|^r\, dw. \label{Far from Tube}
\end{align}

The integration by parts argument for \autoref{Localization in time and space} can be adapted to show that \eqref{Far from Tube} is \(O(\lambda^{- \infty}) \). We use \autoref{GAUSSIAN} to see that \eqref{close to tube} is to leading order
\begin{align}
    C \lambda^{r(m-1)} \int_{\R^{2(m-1)}}e^{-r \lambda \frac{|u|^2}{4 \tau}}\, du  \leq C\lambda^{(r-1)(m-1)}. \label{upper}
\end{align}
Combining these estimates establishes the desired upper bound.

 We now show that this upper bound is sharp. Set
  \begin{align}
    \Phi^{w}_{\chi, \lambda}(z) := \frac{\pcl(z,w)}{\norm{\pcl(\, \cdot\,, w)}_{L^{p}( \pmt)}}.
  \end{align}
Note that $\norm{\Phi^{w}_{\chi, \lambda}}_{L^p(\pmt)} = 1$ and  by \eqref{expansion},
\begin{align}\label{eqn:RATIO}
  \pcl( \Phi^{w}_{\chi, \lambda} )(z)&= \frac{\Pi_{\chi^{2},\lambda}(z,w)}{\norm{\pcl(\, \cdot\,, w)}_{L^{p}( \pmt)}}.
\end{align}

To estimate the numerator of \eqref{eqn:RATIO}, we observe
\begin{align}
  \int_{\pmt}^{} \abs{\pcl(z,w)}^q\,dw   \geq \int_{d\left( T_{\delta}(z), w\right) \leq D \lambda^{ - \frac{1}{2}}}^{}  \abs{\pcl(z,w)} ^q\, dw,
\end{align}
so by applying \autoref{Lower Bound} to the integrand and a using similar argument used to show \eqref{upper}, we may conclude \(\norm{\Pi_{\chi^2,\lambda}(\cdot,w)}_{L^q( \pmt )} \geq C \lambda^{(m - 1)(1 - \frac{1}{q})}\).
Therefore, \(\norm{\Pi_{\chi^{2},\lambda}}_{L^{p}\left( \pmt \right)} \sim \lambda^{(m - 1)( 1 - \frac{1}{q} )}\).

Similarly, the denominator of \eqref{eqn:RATIO} is asymptotically \(\norm{\pcl(\, \cdot\,, w)}_{L^{p}( \pmt)}\sim \lambda^{(m-1)(1- \frac{1}{p})}\).  Together we have
 \begin{align}
  \norm*{ \pcl( \Phi^{w}_{\chi, \lambda} )(z)}_{L^q(\pmt)} \sim \lambda^{(m-1)(1 - \frac{1}{q}) - (m-1)(1 - \frac{1}{p})} = \lambda ^{(m - 1) ( \frac{1}{p} - \frac{1}{q} )},
 \end{align}
which shows $\Phi^{w}_{\chi, \lambda}$ saturates the upper bound.

\subsection{Proof of \autoref{Short window}: norm estimates for \texorpdfstring{$\Pi_{[\lambda,\lambda+1]}$}{}}

A standard argument \cite[Chapter~5]{SoggeFICA} converts the $L^p \to L^q$ estimate for $\pcl$ to that for the projection \(\Pi_{[\lambda, \lambda +1]}\) onto a short spectral interval $[\lambda, \lambda+1]$ as defined in \eqref{eqn:SHORTWINDOW}. We include a proof here for the readers' convenience.

\autoref{Short window} is equivalent to sharpness of the dual inequality
\begin{align}
    \norm{\Pil f}_{L^2( \pmt )} \leq C \lambda^{(m-1)( \frac{1}{p}- \frac{1}{2} )}\norm{f}_{L^p(\pmt )},
\end{align}
which we now establish. For the upper bound, we compute
\begin{align}
    \norm{\Pil f}_{L^2}^2 &= \sum_{\lambda \le \lambda_j < \lambda + 1} \abs{\ip{f}{e_j}}^2\\
    &  \leq C \sum_{j = 1}^\infty\chi(\lambda - \lambda_j)^2 \abs{\ip{f}{e_j}}^2\\
    &=  C \norm{\pcl f}_{L^2}^2\\
    & \leq \lambda^{2(m-1)( \frac{1}{p}- \frac{1}{2} )}\norm{f}_{L^p}^2
\end{align}

To show this upper bound is saturated, fix \(w \in \pmt\) and set \(f_{\chi,\lambda}(z) = \pcl(z,w)\). We compute
\begin{align}
   \norm{\Pil f_{\chi, \lambda}}^2_{L^2} &= \sum_{\lambda \le \lambda_j < \lambda + 1} \abs{\chi(\lambda - \lambda_j )}^2 \abs{e_{\lambda_j}(w)}^2 \norm{e_{\lambda_j}}^2_{L^2} \\
  &\geq C \sum_{\lambda \le \lambda_j < \lambda + 1} \abs{e_j(w)}^2\\
  &\geq \frac{1}{\vol(\pmt)}\int_{\pmt} \sum_{\lambda \le \lambda_j < \lambda + 1}\abs{e_j(z)}^2 \, dz \\
  &\geq C( N(\lambda + 1) - N (\lambda ))\\
  &\sim C\lambda^{m-1},
\end{align}
where \(N(\lambda) = \#\set{j : \lambda_j < \lambda }\) is the eigenvalue counting function.

It follows from the proof of the sharpness of \autoref{theo:MAINTHEO} that \(\norm{f_{\chi, \lambda}}_{L^p} \sim \lambda^{(m-1)( 1 - \frac{1}{p})}\), so
\begin{align}
    \sup_{f \in L^p\left(\pmt\right)}\lambda^{-(m-1)(\frac{1}{p} - \frac{1}{2})}\frac{\norm{\Pil f}_{L^2}}{\norm{f}_{L^p}} &\geq C \lambda^{-\frac{(m-1)}{2}} (N(\lambda +1) - N(\lambda))^{\frac{1}{2}}\\
    &\ge C + o(1).
\end{align}
Taking the \(\limsup\) of both sides we get 
\begin{align}
   \limsup_{\lambda \to \infty}  \sup_{f \in L^p\left(\pmt\right)}\lambda^{-(m-1)(\frac{1}{p} - \frac{1}{2})}\frac{\norm{\Pil f}_{L^2}}{\norm{f}_{L^p}} > 0
\end{align}
which shows that upper bound is sharp.

\section{Proof of \autoref{prop:Approx Reproduce}: complexified Laplace eigenfunctions and  eigenfunctions of \texorpdfstring{\(\pdp\)}{PDP}}
Here we give a proof of \autoref{prop:Approx Reproduce}. In the following we use the parameter \(\mu\) for the frequencies of \(-\Delta\) to distinguish it from the spectral parameter \(\lambda\) used for \(\pdp\). Set 
\begin{align}
  \phitilde(z) = \frac{\phi^{\C}_{\mu}(z)}{\norm{\phi^{\C}_{\mu}}_{L^{2}( \pmt )}}
\end{align}
to be the \(L^{2}(\pmt)\) normalized complexified Laplace eigenfunction.

On one hand, by the first part of \autoref{Unitary pseudo}, we may write $    U({i\tau})^* U(i \tau) = (-\Delta)^{- \frac{m-1}{4}} + R$ for some $R \in \Psi^{- \frac{m+1}{2}}(M)$. It follows that
\begin{align}
  U({i\tau}) \sqrt{-\Delta}^{\frac{m + 1}{2}}U({i\tau})^* \phitilde &= \frac{U({i\tau}) \sqrt{-\Delta}^{\frac{m + 1}{2}}U({i\tau})^*U(i{\tau})\phi_{\mu} }{\sqrt{\ip{U({i\tau})^*U({i\tau})\phi_{\mu}}{\phi_{\mu}}}}\\ \label{Quasimode equation}
  &= \frac{U(i \tau) \sqrt{-\Delta} \phi_\mu + U(i \tau) \sqrt{-\Delta}^{\frac{m+1}{2}} R \phi_\mu}{\sqrt{\ip{U({i\tau})^*U({i\tau})\phi_{\mu}}{\phi_{\mu}}}}\\
  &=\mu \phitilde + O_{L^{2}( \pmt )}( 1 ).
\end{align} 
In the last equality, the first term follows from the definition of complexification and the eigenvalue equation; the second term follows from \(L^2\) boundedness of \((-\Delta)^{\frac{m+1}{4}} R\) as a zeroth order \(\Psi\)DO and \(U(i \tau)\) being a continuous isomorphism \(L^2(M) \to \mathcal{O}^{- \frac{m-1}{4}}(\pmt)\)

On the other hand, by the second part of \autoref{Unitary pseudo}, there exists  $A \in \Psi^{1}( \pmt )$ such that that
\begin{align}
    U({i\tau}) \sqrt{-\Delta}^{\frac{m + 1}{2}}U({i\tau})^* = \Pi_{\tau}A_{}\Pi_{\tau} \quad \text{and} \quad   \sigma(A)\big|_{\Sigma_\tau} = \sigma( D_{\sqrt{\rho} } ).
\end{align}
Therefore, we may write $    \Pi_\tau A \Pi_\tau = \Pi_\tau D_{\sqrt{\rho} }\Pi_\tau + \Pi_\tau B \Pi_\tau$ for some $B \in \Psi^{0}( \pmt )$. It follows from \(L^2\) boundedness of \(B\) that
\begin{align}
    \Pi_\tau D_{\sqrt{\rho}} \Pi_\tau \tilde{\phi}_\mu ^\C = \Pi_\tau A \Pi_\tau\tilde{\phi}_\mu ^\C -  \Pi_\tau B \Pi_\tau \tilde{\phi}_\mu ^\C = \mu \tilde{\phi}_\mu ^\C + O_{L^2(\pmt)}(1).
\end{align}

By a standard theorem giving the distance to the spectrum (see for example \cite[Theorem ~C.11]{Zworski12}), if \(\lambda_j \in \operatorname{spec}(\pdp)\) and \(\mu_j \in \operatorname{spec}(\Pi_{\tau} A \Pi_{\tau})\) then there exists \(M > 0 \) such that for \(\abs{\mu_j - \lambda_j} < M\) for all $j$ sufficiently large. Therefore, we can view the \(\phitilde\) as approximate eigenfunctions for \(\pdp\).

Additionally, we know \(\norm{\tilde{\phi}^{\C}_{\lambda_j} - e_{\lambda_j}}_{L^\infty (\pmt )} = O(\lambda_j^{\frac{m-1}{2}})\) thanks to \cite[Theorem~0.1]{Zelditch20Husimi}. Since \(\norm{\tilde{\phi}^{\C}_{\lambda_j} - e_{\lambda_j}}_{L^2 (\pmt )} = O(1)\), by the log convexity of \(L^{p}\) norms we get \(\norm{\tilde{\phi}_{\lambda_j}^{\C} - e_{\lambda_j}}_{L^p (\pmt )} = O ( \lambda^{(m - 1) ( \frac{1}{2} - \frac{1}{p} )} ) \).

\subsection{Complexified Gaussian beams as extremals: direct computation}\label{sec:SHARP}

In this section, we show that the $L^p$ estimate of \autoref{prop:Approx Reproduce} on complexified Laplace eigenfunctions is saturated by analytic continuations of Gaussian beams on the round $S^2$. We use spherical coordinates
\begin{align}
x = \sin \phi \cos \theta, \quad y = \sin \phi \sin \theta, \quad z = \cos \phi,
\end{align}
where $0 \le \phi \le \pi$ and $0 \le \theta < 2\pi$. The standard spherical harmonics are the joint eigenfunctions
\begin{equation}
-\Delta_{S^2} Y_N^m = N(N+1) Y_N^m, \quad \frac{1}{i}\frac{\partial}{\partial \theta} Y_N^m = m Y_N^m, \quad -N \le m \le N.
\end{equation}
of the spherical Laplacian and the angular momentum operator. The highest weight spherical harmonic (Gaussian beam) is of the form
\begin{align}
Y_N^N(\theta, \phi) = c_N \sin^N(\phi) e^{iN \theta}, \quad c_N = \frac{(-1)^N}{2^N N!} \sqrt{\frac{(2N+1)!}{4\pi}} \sim N^\frac{1}{4}.
\end{align}

It is convenient to transfer the computations of Guillemin--Stenzel \cite{GuilleminStenzel91} from Cartesian coordinates to spherical coordinates. In terms of complexified Cartesian coordinates, the Grauert tube $S^2_\C$ of the sphere is the set
\begin{align}
\set[\big]{(x + i \xi_x, y + i \xi_y, z + i \xi_z) : (x + i \xi_x)^2 + (y + i \xi_y)^2 + (z + i \xi_z)^2 = 1},
\end{align}
and the Grauert tube function is
\begin{align}\label{eqn:TUBE1}
\sqrt{\rho}(x + i\xi_x, y + i \xi_y, z + i\xi_z) = \sinh^{-1}\bra[\Big]{(\xi_x^2 + \xi_y^2 + \xi_z^2)^\frac{1}{2}}.
\end{align}
In terms of complexified spherical coordinates, we have
\begin{align}
    x + i \xi_x &= \sin(\phi + i \xi_\phi) \cos(\theta + i \xi_\theta)\\
    &= \cos(\phi) \sinh(\xi_\phi) \sin(\theta) \sinh(\xi_\theta) + \sin(\phi) \cosh(\xi_\phi) \cos(\theta) \cosh(\xi_\theta) \\
    & \quad + i \bra[\Big]{\cos(\phi) \sinh(\xi_\phi) \cos(\theta) \cosh(\xi_\theta) - \sin(\phi) \cosh(\xi_\phi) \sin(\theta) \sinh(\xi_\theta)}\\
    y + i \xi_y &= \sin(\phi + i \xi_\phi) \sin(\theta + i \xi_\theta)\\
    &= -\cos(\phi) \sinh(\xi_\phi) \cos(\theta) \sinh(\xi_\theta) + \sin(\phi) \cosh(\xi_\phi) \sin(\theta) \cosh(\xi_\theta) \\
    & \quad + i \bra[\Big]{\cos(\phi) \sinh(\xi_\phi) \sin(\theta) \cosh(\xi_\theta) + \sin(\phi) \cosh(\xi_\phi) \cos(\theta) \sinh(\xi_\theta)}\\
    z + i \xi_z &= \cos(\phi) \cosh(\xi_\phi) - i \sin(\phi) \sinh(\xi_\phi).
\end{align}
The formula for $\sqrt{\rho}$ in spherical coordinates is complicated. Since we will be simplifying our expressions by picking special values, it suffices to note
\begin{align}\label{eqn:IMPZETA}
    \xi_x^2 + \xi_y^2 + \xi_z^2 &= \cos^2(\phi) \sinh^2(\xi_\phi) \cosh^2(\xi_\theta) \\
    &\quad  + \sin^2(\phi)(\cosh^2(\xi_\phi)\sinh^2(\xi_\theta) + \sinh^2(\xi_\phi)).
		\end{align}
We also note that the analytically continued highest weight spherical harmonic is of the form
\begin{align} \label{eqn:CPXYLM}
(Y_N^N)^\C (\theta + i\xi_\theta, \phi + i \xi_\phi) = c_N \bra[\Big]{\sin(\phi) \cosh(\xi_\phi) + i \cos(\phi) \sinh(\xi_\phi)}^N e^{iN \theta}e^{-N \xi_\theta}.
\end{align}

To simplify \eqref{eqn:IMPZETA} and \eqref{eqn:CPXYLM}, we fix $\phi = \pi/2$ so that
\begin{align} \label{eqn:TUBE2}
  \xi_x^2 +\xi_y^2 + \xi_z^2 &= \cosh^2(\xi_\phi)\sinh^2(\xi_\theta) + \sinh^2(\xi_\phi),\\
    (Y_N^N)^\C (\theta + i\xi_\theta, i \xi_\phi) &= c_N \cosh^N(\xi_\phi) e^{iN \theta}e^{-N \xi_\theta} \quad (c_N \sim N^\frac{1}{4}).\label{eqn:CPXYLM2}
\end{align}
We additionally set $\tau = 1$, so that \eqref{eqn:TUBE1} and \eqref{eqn:TUBE2} imply the Grauert tube boundary is given by $\sinh^2(1) = \cosh^2(\xi_\phi)\sinh^2(\xi_\theta) + \sinh^2(\xi_\phi)$. Direct computation shows the equality is satisfied whenever
\begin{align}\label{eqn:XITHETA}
    -1 < \xi_\phi < 1 \quad \text{and} \quad \xi_\theta = -\sinh^{-1}\bra[\Big]{\operatorname{sech}(\xi_\phi)\p[\big]{\sinh^2(1) - \sinh^2(\xi_\phi)}^\frac{1}{2}}.
\end{align}
Note that $\xi_\theta < 0$, so
\begin{align}
    \norm{(Y_N^N)^\C}_{L^p}^p &= c_N^p\int_{\partial S^2_1} \abs{\cosh^N(\xi_\phi)}^{p}e^{-Np\xi_\theta} d\xi_\theta d\xi_\phi \\
    &\ge c_N^p \int_{-1}^1\abs{\cosh^N(\xi_\phi)}^{p}e^{N p  \sinh^{-1}\bra[\big]{\operatorname{sech}(\xi_\phi)\p[\big]{\sinh^2(1) - \sinh^2(\xi_\phi)}^\frac{1}{2}}} \, d\xi_\phi \\
    &\ge c_N^p \int_{-1}^1 \abs{\cosh^N(\xi_\phi)}^{p}\,d\xi_\phi \\
    &\ge c_N^p \int_0^1 e^{Np \xi_\phi}\,d\xi_\phi \\
    &\sim N^{\frac{p}{4}}\frac{e^{Np}}{Np}.
\end{align}
In the last line we used $c_N \sim N^{\frac{1}{4}}$. Combined with the universal asymptotics $\norm{(Y_N^N)^\C}_{L^2(\pmt)} = N^{-\frac{1}{4}}e^N (1 + O(N^{-\frac{1}{2}}))$ proved in \cite[Lemma~0.2]{Zelditch20Husimi}, we conclude
\begin{align}
    \frac{\norm{(Y_N^N)^\C}_{L^p(\partial S_1^2)}}{\norm{(Y_N^N)^\C}_{L^2(\partial S_1^2)}} \sim \frac{N^\frac{1}{4}N^{-\frac{1}{p}}e^N}{N^{-\frac{1}{4}}e^N } = N^{\frac{1}{2}- \frac{1}{p}},
\end{align}
showing that \autoref{prop:Approx Reproduce} is sharp.

\subsection{Complexified Gaussian beams as extremals: geometric explanation}
As mentioned earlier, in the real domain Gaussian beams only saturate the low \(L^p\) norms whereas zonal spherical harmonics saturate the high \(L^p\) norms. As in \cite{Zelditch20Husimi}, we give a heuristic symplectic geometry explanation for why complexifications of Gaussian beams are also extremals for high \(L^p\) norms on \(\partial S^2_{\tau}\).

Let \(P\) be the north pole and let \(\frac{\partial }{\partial \theta} \) denote the generator of rotation about the $z$-axis. The zonal spherical harmonics denoted \(Y^{0}_N\) are semiclassical Lagrangian distributions associated to 
\begin{align}
  \Lambda_P = \{g^{t}(S^{*}_P S^2) : t \in \R\} \label{Lagrangian Torus}
\end{align}
Under the natural projection \(S^{*} S^2 \to S^2\)  there is a blowdown singularity at \(P\),  which leads to peaking of sup norms. This is a heuristic explanation for the zonal harmonics saturating high $L^p$ norms in the real domain.

Now let \(E\) denote the equator. The  Gaussian beams \(Y^{N}_N\) associated to \(E\) have wavefront set
\begin{align}
  \Gamma = \{g^{t} ( p, d\theta ): p \in E , \ t \in \R\} 
\end{align}
Th symplectic cone \(\Sigma_{\tau} = ( \pmt, \R^{ +}d \alpha ) \cong  ( S^{*}_{\tau}M, \R^{ +}d\xi )\) from \eqref{eqn:SIGMATAU} is the phase space of the Grauert tube boundary. Under the identification \eqref{eqn:IOTA}, \(\Lambda_{P}\) is a Lagrangian submanifold embedded in \(\pmt\) and no blowdown singularities occur, suggesting that zonal harmonics are longer extremals in the complex domain. Instead, the geodesic flow and the lift of rotations to \(S^{*}S^2\) coincide on \(\Gamma\), and \(\Gamma\) is a singular leaf of the foliations of \(\pmt\) generated by the geodesic flow together with rotations. This singularity suggests that Gaussian extremizes \(L^{p}\) norms in the complex domain.

\bibliographystyle{plain}
\bibliography{References}

\begin{thebibliography}{10}

\bibitem{BlairSogge1}
Matthew~D. Blair and Christopher~D. Sogge.
\newblock Refined and microlocal {K}akeya-{N}ikodym bounds for eigenfunctions
  in two dimensions.
\newblock {\em Anal. PDE}, 8(3):747--764, 2015.

\bibitem{BlairSogge2}
Matthew~D. Blair and Christopher~D. Sogge.
\newblock Refined and microlocal {K}akeya-{N}ikodym bounds of eigenfunctions in
  higher dimensions.
\newblock {\em Comm. Math. Phys.}, 356(2):501--533, 2017.

\bibitem{BleherShiffmanZelditch00}
Pavel Bleher, Bernard Shiffman, and Steve Zelditch.
\newblock Universality and scaling of correlations between zeros on complex
  manifolds.
\newblock {\em Invent. Math.}, 142(2):351--395, 2000.

\bibitem{BoutetGuillemin81}
L.~Boutet~de Monvel and V.~Guillemin.
\newblock {\em The spectral theory of {T}oeplitz operators}, volume~99 of {\em
  Annals of Mathematics Studies}.
\newblock Princeton University Press, Princeton, NJ; University of Tokyo Press,
  Tokyo, 1981.

\bibitem{Boutet78}
Louis Boutet~de Monvel.
\newblock Convergence dans le domaine complexe des s\'eries de fonctions
  propres.
\newblock {\em C. R. Acad. Sci. Paris S\'er. A-B}, 287(13):A855--A856, 1978.

\bibitem{BoutetSjostrand76}
Louis Boutet~de Monvel and Johannes Sj\"ostrand.
\newblock Sur la singularit\'e des noyaux de {B}ergman et de {S}zeg{\H o}.
\newblock pages 123--164. Ast\'erisque, No. 34--35, 1976.

\bibitem{ChangRabinowitz21}
Robert Chang and Abraham Rabinowitz.
\newblock Scaling asymptotics for szeg\"{o} kernels on grauert tubes.
\newblock \url{https://arxiv.org/abs/2107.05105}.

\bibitem{FollandStein}
G.~B. Folland and E.~M. Stein.
\newblock Estimates for the {$\bar \partial _{b}$} complex and analysis on the
  {H}eisenberg group.
\newblock {\em Comm. Pure Appl. Math.}, 27:429--522, 1974.

\bibitem{Folland89harmonic}
Gerald~B. Folland.
\newblock {\em Harmonic analysis in phase space}, volume 122 of {\em Annals of
  Mathematics Studies}.
\newblock Princeton University Press, Princeton, NJ, 1989.

\bibitem{Galkowski}
Jeffrey Galkowski.
\newblock Defect measures of eigenfunctions with maximal {$L^\infty$} growth.
\newblock {\em Ann. Inst. Fourier (Grenoble)}, 69(4):1757--1798, 2019.

\bibitem{GolseLeichtnamStenzel96}
Fran\c{c}ois Golse, Eric Leichtnam, and Matthew Stenzel.
\newblock Intrinsic microlocal analysis and inversion formulae for the heat
  equation on compact real-analytic {R}iemannian manifolds.
\newblock {\em Ann. Sci. \'Ecole Norm. Sup. (4)}, 29(6):669--736, 1996.

\bibitem{GuilleminStenzel91}
Victor Guillemin and Matthew Stenzel.
\newblock Grauert tubes and the homogeneous {M}onge-{A}mp\`ere equation.
\newblock {\em J. Differential Geom.}, 34(2):561--570, 1991.

\bibitem{GuilleminStenzel92}
Victor Guillemin and Matthew Stenzel.
\newblock Grauert tubes and the homogeneous {M}onge-{A}mp\`ere equation. {II}.
\newblock {\em J. Differential Geom.}, 35(3):627--641, 1992.

\bibitem{Hormanderv1}
Lars H{\"o}rmander.
\newblock {\em The analysis of linear partial differential operators. {I}}.
\newblock Classics in Mathematics. Springer-Verlag, Berlin, 2003.
\newblock Distribution theory and Fourier analysis, Reprint of the second
  (1990) edition [Springer, Berlin; MR1065993 (91m:35001a)].

\bibitem{KochTataruZworski}
Herbert Koch, Daniel Tataru, and Maciej Zworski.
\newblock Semiclassical {$L^p$} estimates.
\newblock {\em Ann. Henri Poincar\'{e}}, 8(5):885--916, 2007.

\bibitem{Lebeau}
Gilles Lebeau.
\newblock A proof of a result of {L}. {B}outet de {M}onvel.
\newblock In {\em {A}lgebraic and {A}nalytic {M}icrolocal {A}nalysis}, volume
  269 of {\em Springer Proc. Math. Stat.}, pages 541--574. Springer, Cham,
  2018.

\bibitem{LempertSzoke91}
L\'aszl\'o Lempert and R\'obert Sz{\H o}ke.
\newblock Global solutions of the homogeneous complex {M}onge-{A}mp\`ere
  equation and complex structures on the tangent bundle of {R}iemannian
  manifolds.
\newblock {\em Math. Ann.}, 290(4):689--712, 1991.

\bibitem{LempertSzoke01}
L\'aszl\'o Lempert and R\'obert Sz{\H o}ke.
\newblock The tangent bundle of an almost complex manifold.
\newblock {\em Canad. Math. Bull.}, 44(1):70--79, 2001.

\bibitem{Paoletti20122}
Roberto Paoletti.
\newblock Local trace formulae and scaling asymptotics for general quantized
  {H}amiltonian flows.
\newblock {\em J. Math. Phys.}, 53(2):023501, 22, 2012.

\bibitem{Paoletti20121}
Roberto Paoletti.
\newblock Scaling asymptotics for quantized {H}amiltonian flows.
\newblock {\em Internat. J. Math.}, 23(10):1250102, 25, 2012.

\bibitem{Paoletti2014}
Roberto Paoletti.
\newblock Local scaling asymptotics in phase space and time in
  {B}erezin-{T}oeplitz quantization.
\newblock {\em Internat. J. Math.}, 25(6):1450060, 40, 2014.

\bibitem{ShiffmanZelditch02}
Bernard Shiffman and Steve Zelditch.
\newblock Asymptotics of almost holomorphic sections of ample line bundles on
  symplectic manifolds.
\newblock {\em J. Reine Angew. Math.}, 544:181--222, 2002.

\bibitem{ShiffmanZelditch03}
Bernard Shiffman and Steve Zelditch.
\newblock Random polynomials of high degree and {L}evy concentration of
  measure.
\newblock {\em Asian J. Math.}, 7(4):627--646, 2003.

\bibitem{SoggeLP}
Christopher~D. Sogge.
\newblock Concerning the {$L^p$} norm of spectral clusters for second-order
  elliptic operators on compact manifolds.
\newblock {\em J. Funct. Anal.}, 77(1):123--138, 1988.

\bibitem{SoggeFICA}
Christopher~D. Sogge.
\newblock {\em Fourier integrals in classical analysis}, volume 210 of {\em
  Cambridge Tracts in Mathematics}.
\newblock Cambridge University Press, Cambridge, second edition, 2017.

\bibitem{SoggeTothZelditch}
Christopher~D. Sogge, John~A. Toth, and Steve Zelditch.
\newblock About the blowup of quasimodes on {R}iemannian manifolds.
\newblock {\em J. Geom. Anal.}, 21(1):150--173, 2011.

\bibitem{SoggeZelditchMax}
Christopher~D. Sogge and Steve Zelditch.
\newblock Riemannian manifolds with maximal eigenfunction growth.
\newblock {\em Duke Math. J.}, 114(3):387--437, 2002.

\bibitem{SoggeZelditchFocal1}
Christopher~D. Sogge and Steve Zelditch.
\newblock Focal points and sup-norms of eigenfunctions.
\newblock {\em Rev. Mat. Iberoam.}, 32(3):971--994, 2016.

\bibitem{SoggeZelditchFocal2}
Christopher~D. Sogge and Steve Zelditch.
\newblock Focal points and sup-norms of eigenfunctions {II}: the
  two-dimensional case.
\newblock {\em Rev. Mat. Iberoam.}, 32(3):995--999, 2016.

\bibitem{Daubechies80}
Kurt~Bernardo Wolf.
\newblock Canonical transforms. {IV}. {H}yperbolic transforms: continuous
  series of {${\rm SL}(2,\,{\bf R})$} representations.
\newblock {\em J. Math. Phys.}, 21(4):680--688, 1980.

\bibitem{Zelditch20Husimi}
Steve Zelditch.
\newblock {$L^\infty$} norms of {H}usimi distributions of eigenfunctions.
\newblock \url{https://arxiv.org/abs/2010.13212}.

\bibitem{Zelditch97index}
Steve Zelditch.
\newblock Index and dynamics of quantized contact transformations.
\newblock {\em Ann. Inst. Fourier (Grenoble)}, 47(1):305--363, 1997.

\bibitem{Zelditch07complex}
Steve Zelditch.
\newblock Complex zeros of real ergodic eigenfunctions.
\newblock {\em Invent. Math.}, 167(2):419--443, 2007.

\bibitem{Zelditch12potential}
Steve Zelditch.
\newblock Pluri-potential theory on {G}rauert tubes of real analytic
  {R}iemannian manifolds, {I}.
\newblock In {\em Spectral geometry}, volume~84 of {\em Proc. Sympos. Pure
  Math.}, pages 299--339. Amer. Math. Soc., Providence, RI, 2012.

\bibitem{Zelditch14intersect}
Steve Zelditch.
\newblock Ergodicity and intersections of nodal sets and geodesics on real
  analytic surfaces.
\newblock {\em J. Differential Geom.}, 96(2):305--351, 2014.

\bibitem{ZelditchZhou2018}
Steve Zelditch and Peng Zhou.
\newblock Pointwise {W}eyl law for partial {B}ergman kernels.
\newblock In {\em Algebraic and analytic microlocal analysis}, volume 269 of
  {\em Springer Proc. Math. Stat.}, pages 589--634. Springer, Cham, 2018.

\bibitem{ZelditchZhou20191}
Steve Zelditch and Peng Zhou.
\newblock Central limit theorem for spectral partial {B}ergman kernels.
\newblock {\em Geom. Topol.}, 23(4):1961--2004, 2019.

\bibitem{ZelditchZhou20192}
Steve Zelditch and Peng Zhou.
\newblock Interface asymptotics of partial {B}ergman kernels around a critical
  level.
\newblock {\em Ark. Mat.}, 57(2):471--492, 2019.

\bibitem{Zworski12}
Maciej Zworski.
\newblock {\em Semiclassical analysis}, volume 138 of {\em Graduate Studies in
  Mathematics}.
\newblock American Mathematical Society, Providence, RI, 2012.

\end{thebibliography}
\end{document}